\documentclass[11pt,a4paper]{article}
\usepackage{mathrsfs}
\usepackage{multirow}
\usepackage{epsfig}
\usepackage{epstopdf}
\usepackage{stmaryrd}
\usepackage[T1]{fontenc}
\usepackage{geometry}
\usepackage{amsbsy,amsmath,latexsym,amsfonts, epsfig, color, authblk, amssymb, graphics, bm}
\usepackage{epsf,slidesec,epic,eepic}
\usepackage{fancybox}
\usepackage{fancyhdr}
\usepackage{setspace}
\usepackage{nccmath}
\usepackage{cases}
\usepackage{extarrows}
\usepackage{multirow}
\usepackage{algorithm}
\usepackage{algorithmic}
\usepackage{rotating}
\usepackage{emptypage}
\usepackage{graphicx}
\DeclareGraphicsExtensions{.eps,.ps,.jpg,.bmp}

\usepackage[colorlinks, citecolor=blue]{hyperref}

\newtheorem{theorem}{Theorem}[section]
\newtheorem{lemma}{Lemma}[section]

\newtheorem{definition}{Definition}[section]
\newtheorem{example}{Example}[section]
\newtheorem{proposition}{Proposition}[section]
\newtheorem{corollary}{Corollary}[section]
\newtheorem{assumption}{Assumption}
\newtheorem{remark}{Remark}[section]

\newtheorem{alemma}{Lemma}

\newenvironment{aproof}{{\noindent}}{\hfill$\Box$\medskip}
\newenvironment{proof}{{\noindent \bf Proof:}}{\hfill$\Box$\medskip}
\definecolor{lred}{rgb}{1,0.8,0.8}
\definecolor{lblue}{rgb}{0.8,0.8,1}
\definecolor{dred}{rgb}{0.6,0,0}
\definecolor{dblue}{rgb}{0,0,0.5}
\definecolor{dgreen}{rgb}{0,0.5,0.5}

 \title{A proximal MM method for the zero-norm regularized PLQ composite optimization problem}

 \author{Dongdong Zhang\footnote{School of Mathematics, South China University of Technology, Guangzhou, China.},\
 \ Shaohua Pan\footnote{(shhpan@scut.edu.cn) School of Mathematics, South China University of Technology,
 China.}\ \ and\ \ Shujun Bi\footnote{(bishj@scut.edu.cn) School of Mathematics, South China University of Technology,
 China.}}

\begin{document}

 \maketitle

 \begin{abstract}
  This paper is concerned with a class of zero-norm regularized piecewise
  linear-quadratic (PLQ) composite minimization problems, which covers the zero-norm
  regularized $\ell_1$-loss minimization problem as a special case.
  For this class of nonconvex nonsmooth problems, we show that its equivalent
  MPEC reformulation is partially calm on the set of global optima and make use of
  this property to derive a family of equivalent DC surrogates. Then, we propose
  a proximal majorization-minimization (MM) method, a convex relaxation approach
  not in the DC algorithm framework, for solving one of the DC surrogates which is
  a semiconvex PLQ minimization problem involving three nonsmooth terms. For this method,
  we establish its global convergence and linear rate of convergence,
  and under suitable conditions show that the limit of the generated sequence is
  not only a local optimum but also a good critical point in a statistical sense.
  Numerical experiments are conducted with synthetic and real data for
  the proximal MM method with the subproblems solved by a dual semismooth Newton method
  to confirm our theoretical findings, and numerical comparisons
  with a convergent indefinite-proximal ADMM for the partially smoothed DC surrogate
  verify its superiority in the quality of solutions and computing time.
 \end{abstract}


%-----------------------------------------------------------------------------------------------
 \section{Introduction}\label{sec1}

  Let $f\!:\mathbb{R}^n\!\to\mathbb{R}_{+}$ be a piecewise linear-quadratic
  nonsmooth convex function, $A\!\in\mathbb{R}^{n\times p}$ and $b\in\!\mathbb{R}^n$
  be the given matrix and vector, and $\Omega\subseteq\!\mathbb{R}^p$ be
  a nonempty polyhedral set. We are interested in the following zero-norm regularized
  composite minimization problem
  \begin{equation}\label{prob}
   \min_{x\in\Omega}\Big\{f(Ax-b)+\frac{\mu}{2}\|x\|^2+\nu\|x\|_0\Big\}
  \end{equation}
  where $\nu>0$ is a regularization parameter, $\|\cdot\|_0$ denotes
  the zero-norm (cardinality) of vectors, and $\mu>0$ is a small regularization
  parameter. The term $\frac{1}{2}\mu\|x\|^2$ is introduced to ensure that the objective
  function of \eqref{prob} is coercive and then has a nonempty global optimum set.
  For convenience, we denote the sum of the first two functions by
 \begin{equation}\label{Fmu}
  F_{\mu}(x):=f(Ax-b)+(\mu/2)\|x\|^2\quad\ \forall x\in\mathbb{R}^p.
 \end{equation}
  Since the zero-norm is the root to produce sparse solutions,
  the problem \eqref{prob} has been found to have wide applications in
  a host of scientific and engineering problems such as regression and
  variable selection in statistics (see, e.g., \cite{Tibshirani96,Fan01}),
  compressed sensing \cite{DL06} and source separation \cite{Bruckstein09}
  in signal processing, imaging decomposition \cite{Soubies15} in image science,
  feature selection and classification in statistical learning
  \cite{Bradley98,Weston03}, and so on. In particular, the nonsmooth PLQ
  loss $f(Ax-b)$ makes the problem \eqref{prob} arise frequently from robust models;
  for example, when $f(z)=\frac{1}{n}\sum_{i=1}^n\theta(z_i)$
  with $\theta(t)=|t|$ for $t\in\mathbb{R}$, it becomes the popular sparsity
  regularized $\ell_1$-loss minimization in robust sparse recovery \cite{Wright10,LiuYGS18}
  and high-dimensional robust statistics \cite{Wang07,Wang13};
  and when $\theta(t)=(\tau-\mathbb{I}_{\{t\le 0\}})t$ for some $\tau\in(0,1)$,
  it reduces to the sparsity regularized check-loss minimization that is often
  used to monitor the heteroscedasticity of high-dimensional data \cite{WuLiu09,Wang12}.

  \medskip

  Owing to the combinatorial property of the zero-norm function,
  the problem \eqref{prob} is generally NP-hard and it is impractical
  to seek a global minimum with an algorithm of polynomial-time complexity.
  For this class of nonconvex nonsmooth problems, a common way is to
  use the convex relaxation technique to achieve a desirable solution in
  a statistical sense. The $\ell_1$-norm convex relaxation, as a popular
  relaxation method, has witnessed considerable progress in theory and computation
  since the early works \cite{DS89,Tibshirani96}. Though the $\ell_1$-norm
  is the convex envelope of the zero-norm in the $\ell_\infty$-norm unit ball,
  its ability to promote sparsity is weak especially in a complicated
  constraint set, say, the simplex set. Inspired by this, many nonconvex
  surrogates have been proposed for the zero-norm function, which include
  the non-Lipschitz $\ell_p\ (0<p<1)$ surrogate \cite{Chartrand07,Chen10},
  smooth concave approximation \cite{Bradley98,Rinaldi10,Weston03},
  and the folded concave functions such as SCAD \cite{Fan01} and MCP \cite{Zhang10}.
  All of these nonconvex surrogates are proposed from the primal viewpoint
  and the surrogate problems associated to the first two classes are
  only an approximation of \eqref{prob}. Although Soubies et al. \cite{Soubies17}
  proposed a class of exact continuous relaxation for the $\ell_2$-$\ell_0$ minimization,
  their proof depends on the structure of the least-square loss function and
  it is not clear whether they are exact or not for the problem \eqref{prob}.

  \medskip

  One contribution of this work is to show that an equivalent MPEC of \eqref{prob}
  is partially calm on the set of global optima, thereby obtaining a family of
  equivalent DC surrogates from a primal-dual viewpoint.
  The calmness of a mathematical programming problem at a solution point was
  originally introduced by Clarke \cite{Clarke76}, and received active study from many
  researchers in the past several decades (see, e.g., \cite{Burke911,Yezhu95,Yezhu97}).
  Among others, Ye and Zhu \cite{Yezhu95,Yezhu97} extended it to the partial calmness
  at a solution point. Inspired by these works, Liu et al. \cite{LiuBiPan18} recently
  studied the partial calmness on the global optimum set for the equivalent
  MPECs of zero-norm and rank optimization problems so as to achieve their global
  exact penalty. By \cite[Theorem 3.2]{LiuBiPan18}, if a special structure is
  imposed on the set $\Omega$, such an MPEC is indeed partially calm on
  the set of global optima, but this structure is very restricted.
  Here we achieve this crucial property without any restriction on $\Omega$
  by constructing an appropriate multifunction and using the upper Lipschitz
  continuity of the polyhedral multifunction due to Robinson \cite{Robinson81}.
  Also, by combining this result with \cite[Appendix B]{LiuBiPan18}, we conclude that
  the SCAD is a member of this family of equivalent DC surrogates.
  Although Le Thi et al. \cite{LeThi15} ever derived an equivalent DC surrogate
  for the zero-norm regularized problem from a primal-dual viewpoint,
  they required $\Omega$ to be a compact box set, and their surrogate
  has a great difference from ours; see Section \ref{sec3}.

  \medskip

  Then, inspired by the work \cite{Tang19}, we propose a proximal
  majorization-minimization (MM) method for solving one of the equivalent
  DC surrogate models, a semiconvex PLQ minimization problem involving
  three nonsmooth terms, by using a tighter majorization of the DC
  surrogate function. The other contribution of this work is to
  establish the global convergence and linear rate of convergence
  for the proximal MM method, and to demonstrate when the limit of
  the generated sequence is locally optimal to the surrogate problem
  and the origin problem \eqref{prob}, respectively. In particular,
  for the scenario where $f(z)=\frac{1}{n}\sum_{i=1}^n\theta(z_i)$ with
  an appropriate convex $\theta\!:\mathbb{R}\to\mathbb{R}_{+}$ and the data $(b,A)$
  comes from a linear observation model, we also derive an error bound to
  the true vector for the limit of the generated sequence, which shows that
  the limit is good from a statistical perspective.
  Numerical experiments are conducted with some synthetic and real data
  for the proposed proximal MM with the subproblems solved by a dual semismooth
  Newton method (PMMSN for short) to confirm our theoretical findings.
  In particular, we compare the performance of PMMSN with the performance of
  a globally convergent indefinite-proximal ADMM (iPADMM) proposed for solving
  the partially smoothed DC surrogate problem. Numerical comparisons indicate
  that PMMSN has an advantage in the quality of solutions and computing time
  for most of test examples, whereas iPADMM depends on the choice of the smoothing
  parameter that is very sensitive to the data.

  \medskip

  It is worthwhile to emphasize that for optimization models involving a smooth
  loss term and a nonconvex surrogate of the zero-norm, there are some works to
  investigate the error bounds of their stationary points to the true vector
  (see, e.g., \cite{Loh15,Cao18}) or the oracle property of a local optimum
  yielded by a specific algorithm \cite{Fan14b}, but to the best of our knowledge,
  for optimization models involving a nonsmooth convex loss and such an equivalent
  DC surrogate, there are few works on the statistical error bound
  of the critical point yielded by an algorithm. The optimization model
  in \cite{Tang19} involves a square-root loss and such a DC surrogate,
  but the local optimality and statistical error bound of the obtained
  critical point was not discussed. For the box constrained zero-norm
  regularized nonsmooth convex loss minimization, Bian and Chen \cite{Bian19}
  recently presented an exact continuous relaxation and proposed a smoothing
  proximal gradient algorithm for finding a lifted stationary point of
  the relaxation model, but they did not provide a statistical error bound
  for the lifted stationary point. Also, as will be shown in Section \ref{sec3},
  their continuous relaxation model is actually a member of our equivalent DC surrogates.

  \medskip
  \noindent
  {\bf Notation:} In this paper, $I$ and $e$ denote an identity matrix
  and a vector of all ones, respectively, whose dimensions are known from the context.
  For any matrix $X=[X_{ij}]_{n\times p}$, $\|X\|,\|X\|_{\infty}$ and
  $\!\interleave X\interleave_1$ denote the spectral norm, elementwise maximum
  norm and column sum norm of $A$, respectively, i.e.,
  $\!\interleave X\interleave_1\!=\max_{1\le j\le p}\sum_{i=1}^n|X_{ij}|$,
  and for given index sets $\mathcal{I}\subseteq\{1,\ldots,n\}$ and
  $\mathcal{J}\subseteq\{1,\ldots,p\}$, $X_{\mathcal{I}\cdot}$ and $X_{\cdot\mathcal{J}}$
  denote the matrix consisting of those rows $X_{i\cdot}$ of $X$ with $i\in\mathcal{I}$
  and those columns $X_{\cdot j}$ of $X$ with $j\in\mathcal{J}$, respectively.
  For a set $S$, ${\rm conv}(S)$ means the convex hull of $S$ and $\delta_S$
  denotes the indicator function of $S$, i.e., $\delta_S(x)=0$ if $x\in S$
  and $+\infty$ otherwise.
  For given vectors $a,b\in\mathbb{R}^p$ with $a_i\le b_i$
  for each $i=1,\ldots,p$, the notation $[a,b]$ denotes the box set.
  For a vector $x$, $x_{\rm nz}$ represents the smallest nonzero entry of $x$
  and $|x|^{\downarrow}$ denotes the vector obtained by arranging the entries
  of $|x|$ in a nonincreasing order. For an extended real-valued $f\!:\mathbb{R}^p\to[-\infty,+\infty]$,
  we say that $f$ is proper if $f(x)>-\infty$ for all $x\in\mathbb{R}^p$
  and ${\rm dom}\,f\!:=\{x\in\mathbb{R}^p\ |\ f(x)<\infty\}\ne\emptyset$,
  and denote by $f^*$ the conjugate of $f$, i.e.,
  $f^*(x^*)\!:=\sup_{x\in\mathbb{R}^p}\{\langle x^*,x\rangle-f(x)\}$.
  For a lower semicontinuous (lsc) function $f\!:\mathbb{R}^p\to(-\infty,+\infty]$
  and a parameter $\gamma>0$, $\mathcal{P}_{\gamma}f$ and $e_{\gamma}f$
  denote the proximal mapping and Moreau envelope of $f$, respectively,
  defined as
 \[
  \mathcal{P}_{\gamma}f(x):=\mathop{\arg\min}_{z\in\mathbb{R}^p}\Big\{f(z)+\frac{1}{2\gamma}\|z-x\|^2\Big\}
  \ {\rm and}\ e_{\gamma}f(x):=\min_{z\in\mathbb{R}^p}\Big\{f(z)+\frac{1}{2\gamma}\|z-x\|^2\Big\}.
 \]
 When $f$ is convex, $\mathcal{P}_{\gamma}f\!:\mathbb{R}^p\to\mathbb{R}^p$
 is a Lipschitzian mapping with Lipschitz constant $1$, and $e_{\gamma}f$ is
 a smooth convex function with $\nabla e_{\gamma}f(x)=\gamma^{-1}(x-\mathcal{P}_{\gamma}f(x))$.
%-----------------------------------------------------------------------------------
 \section{Preliminaries}\label{sec2}

 First, we recall the concepts of the proximal, regular and limiting subdifferentials
 from \cite[Definition 8.45 \& 8.3]{RW98} and the definition of the subderivative
 and second subderivative from \cite[Definition 8.1 \& 13.3]{RW98} for
 an extended real-valued function.
%--------------------------------------------------------------------------------
 \subsection{Generalized subdifferentials and subderivatives}\label{sec2.1}
%-----------------------------------------------------------------------------------
  \begin{definition}\label{Gsubdiff-def}
  Consider a function $f\!:\mathbb{R}^p\to[-\infty,+\infty]$ and a point
  $x$ with $f(x)$ finite. The proximal subdifferential of $f$ at $x$, denoted by
  $\widetilde{\partial}\!f(x)$, is defined as
  \[
    \widetilde{\partial}\!f(x):=\bigg\{v\in\mathbb{R}^p\ \big|\
    \liminf_{x\ne x'\to x}
    \frac{f(x')-f(x)-\langle v,x'-x\rangle}{\|x'-x\|^2}>-\infty\bigg\};
  \]
  the regular subdifferential of $f$ at $x$, denoted by $\widehat{\partial}\!f(x)$, is defined as
  \[
    \widehat{\partial}\!f(x):=\bigg\{v\in\mathbb{R}^p\ \big|\
    \liminf_{x\ne x'\to x}\frac{f(x')-f(x)-\langle v,x'-x\rangle}{\|x'-x\|}\ge 0\bigg\};
  \]
  and the (limiting) subdifferential of $f$ at $x$, denoted by $\partial\!f(x)$, is defined as
  \[
    \partial\!f(x):=\Big\{v\in\mathbb{R}^p\ |\  \exists\,x^k\to x\ {\rm with}\ f(x^k)\to f(x)\ {\rm and}\
    v^k\in\widehat{\partial}\!f(x^k)\ {\rm with}\ v^k\to v\Big\}.
  \]
 \end{definition}
%-------------------------------------------------------------------------------
 \begin{remark}\label{remark-Gsubdiff}
  {\bf(i)} At each $x\in{\rm dom}f$, the sets $\widetilde{\partial}f(x)$ and
  $\widehat{\partial}\!f(x)$ are always closed convex, $\partial\!f(x)$ is
  closed but generally nonconvex, and they satisfy
  $\widetilde{\partial}f(x)\subseteq\widehat{\partial}\!f(x)\subseteq\partial\!f(x)$.
  These inclusions may be strict when $f$ is nonconvex, and when $f$ is convex,
  they all reduce to the subdifferential of $f$ at $x$ in the sense of
  convex analysis \cite{Roc70}.
%
%  \medskip
%  \noindent
%  {\bf(ii)} Let $\{(x^k,v^k)\}_{k\in\mathbb{N}}$ be a sequence with
%  $v^k\in\partial f(x^k)$ that converges to $(x,v)$ as $k\to\infty$.
%  By invoking Definition \ref{Gsubdiff-def}, if $f(x^k)\to f(x)$ as $k\to\infty$,
%  then $v\in\partial f(x)$.

  \medskip
  \noindent
  {\bf(ii)} The point $\overline{x}$ at which $0\in\partial\!f(\overline{x})$
  (respectively, $0\in\widetilde{\partial}\!f(\overline{x})$ and
  $0\in\widehat{\partial}\!f(\overline{x})$) is called a limiting
  (respectively, proximal and regular) critical point of $f$.
  By \cite[Theorem 10.1]{RW98}, a local minimizer of $f$ is necessarily
  a regular critical point of $f$, and then a limiting critical point.
  However, the converse may not hold; for example, the function
  $f(t)=-|t|+t$ for $t\in\mathbb{R}$ satisfies $0\in\partial\!f(0)$,
  but $0$ is not a local minimizer of $\min_{t\in\mathbb{R}}f(t)$.

 \medskip
 \noindent
 {\bf(iii)} Recall that a function $f\!:\mathbb{R}^p\to[-\infty,+\infty]$
 is said to be semiconvex if there exists a constant $\gamma>0$ such that
 $x\mapsto f(x)+\frac{\gamma}{2}\|x\|^2$ is convex, and the smallest
 of all such $\gamma$ is called the semi-convex modulus of $f$.
 For this $f$,
 \(
  \widetilde{\partial}\!f(x)=\widehat{\partial}\!f(x)=\partial\!f(x)
 \)
 at all $x\in {\rm dom}f$.
 \end{remark}
%---------------------------------------------------------------------------------
 \begin{definition}
 For a function $f\!:\mathbb{R}^p\to[-\infty,+\infty]$, a point $\overline{x}$ with
 $f(\overline{x})$ finite and any $v\in\mathbb{R}^p$, the subderivative function $df(\overline{x})\!:\mathbb{R}^p
 \to[-\infty,+\infty]$ is defined by
 \[
   df(\overline{x})(w):=\liminf_{\tau\downarrow0,\,w'\to w}
   \frac{f(\overline{x}+\tau w')-f(\overline{x})}{\tau},
 \]
 while the second subderivative of $f$ at $\overline{x}$ for $v$ and $w$ is defined by
 \[
   d^2f(\overline{x}|v)(w):=\liminf_{\tau\downarrow 0,\,w'\to w}
   \frac{f(\overline{x}+\tau w')-f(\overline{x})-\tau\langle v,w'\rangle}{\frac{1}{2}\tau^2}.
 \]
 \end{definition}
%--------------------------------------------------------------------------------------
 \subsection{Semismoothness of locally Lipschitzian mappings}\label{sec2.2}

  Semismoothness was originally introduced by Mifflin \cite{Mifflin77} for functionals,
  and Qi and Sun \cite{QiSun93} later developed the class of vector semismooth functions.
  Before introducing the semismoothness, we recall the Clarke Jacobian
  of a locally Lipschitzian mapping.
%--------------------------------------------------------------------------
 \begin{definition}(see \cite[Definition 2.6.1]{Clarke83})
  Let $F\!:\mathcal{O}\subseteq\mathbb{R}^n\to\mathbb{R}^m$ be a locally Lipschitzian
  mapping defined on an open set $\mathcal{O}$. Denote by
  $D_F\subseteq\mathcal{O}$ the set of points where $F$ is Fr\'{e}chet differentiable
  and by $F'(z)\in\mathbb{R}^{m\times n}$ the Jacobian of $F$ at $z\in D_F$.
  For any given $\overline{z}\in\mathcal{O}$, the Clarke (generalized) Jacobian
  of $F$ at $\overline{z}$ is defined as
  \[
    \partial_CF(\overline{z}):={\rm conv}\Big\{\lim_{k\to\infty}F'(z^k)\ |\
    \{z^k\}\subseteq D_F\ {\rm with}\ \lim_{k\to\infty}z^k=\overline{z}\Big\}.
  \]
 \end{definition}
%-----------------------------------------------------------------------------Definition
 \begin{definition}\label{def1}(see \cite{QiSun93,SunSun02})
  Let $F\!:\mathcal{O}\subseteq\mathbb{R}^n\to\mathbb{R}^m$
  be a locally Lipschitzian mapping on an open set $\mathcal{O}$.
  The mapping $F$ is said to be semismooth at a point $x\in\mathcal{O}$ if
  $F$ is directionally differentiable at $x$ and for any $\Delta x\to 0$ and
  $V\in\partial_{C}F(x+\Delta x)$,
  \[
    F(x+\Delta x)-F(x)-V\Delta x =o(\|\Delta x\|);
  \]
  and $F$ is said to be strongly semismooth at $x$ if $F$ is semismooth
  at $x$ and for any $\Delta x\to 0$,
  \[
    F(x+\Delta x)-F(x)-V\Delta x =O(\|\Delta x\|^{2}).
  \]
  The mapping $F$ is said to be a semismooth (respectively, strongly semismooth)
  function on $\mathcal{O}$ if it is semismooth (respectively, strongly semismooth)
  everywhere in $\mathcal{O}$.
 \end{definition}
%----------------------------------------------------------------------------------
 \begin{lemma}\label{CJacobi-h}
  Given $\omega\in\mathbb{R}_{+}^p$ and $\mu\ge 0$, let $h(x):=\|\omega\circ x\|_1+\frac{1}{2}\mu\|x\|^2$
  for $x\in\mathbb{R}^p$. Then, $\mathcal{P}_{\gamma^{-1}}h(z)\!=\frac{\gamma}{\gamma+\mu}{\rm sign}(z)\circ\max(|z|-\omega/\gamma,0)$ for $z\in\mathbb{R}^p$ and
  its Clarke Jacobian is given by
  \begin{equation*}
  \partial_C(\mathcal{P}_{\gamma^{-1}}h)(z)
  =\!\Big\{{\rm Diag}(v_1,\ldots,v_n)\ |\ v_i=\frac{\gamma}{\gamma\!+\!\mu}\ {\rm if}\ |z_i|>\frac{\omega_i}{\gamma},\,
     {\rm otherwise}\,v_i\in\big[0,\frac{\gamma}{\gamma\!+\!\mu}\big]\Big\}.
  \end{equation*}
 \end{lemma}

 Notice that $\mathcal{P}_{\gamma^{-1}}h$ in Lemma \ref{CJacobi-h}
 is piecewise affine. By \cite[Proposition 7.4.7]{Facchinei03}
 every piecewise affine mapping is strongly semismooth.
 Hence, it is strongly semismooth.

%--------------------------------------------------------------------------------
 \section{Equivalent DC surrogates of problem \eqref{prob}}\label{sec3}

  Let $\mathscr{L}$ be the family of proper lsc functions
  $\phi\!:\mathbb{R}\to(-\infty,+\infty]$ with ${\rm int}({\rm dom}\,\phi)
  \supseteq[0,1]$ which are convex in the interval $[0,1]$ and
  satisfy the following conditions
  \begin{equation}\label{phi-assump}
   1>t^*:=\mathop{\arg\min}_{0\le t\le 1}\phi(t),\ \phi(t^*)=0
   \ \ {\rm and}\ \ \phi(1)=1.
  \end{equation}
  Since $\phi(1)=1$ and $t^*$ is the unique minimizer of $\phi$ over $[0,1]$,
  for any $z\in\mathbb{R}^p$ we have
  \begin{equation*}
   \|z\|_0=\min_{w\in\mathbb{R}^p}\Big\{{\textstyle\sum_{i=1}^p}\phi(w_i)\quad\mbox{s.t.}
   \ \ \big\langle e-w,|z|\big\rangle=0,\,0\le w\leq e\Big\}.
  \end{equation*}
  As far as we know, such a characterization for $z\in\mathbb{R}_{+}^p$
  with $\phi(t)=t$ first appeared in \cite{Mangasarian96}. This implies that the zero-norm
  regularized problem \eqref{prob} is equivalent to the problem
 \begin{equation}\label{Eprob}
  \min_{x,w\in\mathbb{R}^p}\Big\{F_{\mu}(x)+\nu\textstyle{\sum_{i=1}^p}\phi(w_i)\quad\mbox{s.t.}\ \
  \langle e-w,|x|\rangle=0,\,x\in\Omega,\,w\in[0,e]\Big\}
 \end{equation}
 in the following sense: if $x^*$ is globally optimal to the problem \eqref{prob},
 then $(x^*\!,{\rm sign}(|x^*|))$ is a global optimal solution of the problem \eqref{Eprob},
 and conversely, if $(x^*,w^*)$ is a global optimal solution of \eqref{Eprob},
 then $x^*$ is globally optimal to \eqref{prob}. The problem \eqref{Eprob} is a mathematical
 program with an equilibrium constraint $e-w\ge 0,|x|\ge 0,\langle e-w,|x|\rangle=0$.
 Although the MPEC is known to be very tough, in this section we can establish that
 the MPEC \eqref{Eprob} is partially calm over its global optimal solution set,
 denoted by $\mathcal{S}^*$, and employ its relation with exact penalization
 to derive a class of equivalent DC surrogates for \eqref{prob}.

 \medskip

 As mentioned in Section \ref{sec1}, when $\Omega$ has a special structure,
 the MPEC \eqref{Eprob} is partially calm over the set $\mathcal{S}^*$
 by \cite[Theorem 3.2]{LiuBiPan18}, but the required structure is
 very restricted. Here, we remove such a restriction and obtain the partial
 calmness of \eqref{Eprob} over $\mathcal{S}^*$.
%-----------------------------------------------------------------------------------------
 \begin{proposition}\label{partial-calm}
  The problem \eqref{Eprob} is partially calm over the optimal solution set $\mathcal{S}^*$.
 \end{proposition}
 \begin{proof}
  For each $\epsilon\in\!\mathbb{R}$, define the partial perturbation
  for the feasible set of \eqref{Eprob} by
 \[
  \mathcal{S}_\epsilon:=\big\{(x,w)\in\Omega\times[0,e]\ |\
  \langle e-w,|x|\rangle=\epsilon\big\}.
 \]
 Fix an arbitrary $(x^*,w^*)\in\mathcal{S}^*$. Notice that the objective
 function of the problem \eqref{Eprob} is continuous. By \cite[Remark 2.3]{Yezhu97},
 its partial calmness at $(x^*,w^*)$ is equivalent to the existence of $\delta>0$
 and $\rho>0$ such that for all $\epsilon\in\mathbb{R}_{+}$ and
 all $(x,w)\in\mathbb{B}((x^*,w^*),\delta)\cap\mathcal{S}_{\epsilon}$,
 \begin{equation}\label{aim-ineq}
   \big[F_{\mu}(x) +\nu{\textstyle\sum_{i=1}^p}\phi(w_i)\big]
   -\big[F_{\mu}(x^*) +\nu{\textstyle\sum_{i=1}^p}\phi(w_i^*)\big]
   +\rho\nu\langle e-w,|x|\rangle\ge 0.
 \end{equation}

 Observe that the solution set $\mathcal{S}^*$ is compact. There necessarily exists
 $R>0$ such that $\bigcup_{x^*\in\mathcal{S}^*}\mathbb{B}(x^*,1/2)
 \subseteq\mathbb{B}_{R}:=\{x\in\mathbb{R}^p\ |\ \|x\|_\infty\le R\}$.
 It is easy to check that the function $F_{\mu}$ is Lipschitzian
 relative to $\mathbb{B}_{R}$. We denote by $L_{F_{\mu}}$
 the Lipschitz constant of $F_{\mu}$ over the set $\mathbb{B}_{R}$.
 For each $k\in\{1,2,\ldots,p\}$, define the multifunction
 $\Gamma^k\!:\mathbb{R}\rightrightarrows\mathbb{R}^p$ by
  \[
    \Gamma^k(\tau):=\Big\{x\in\Omega\cap\mathbb{B}_{R}\ |\ \|x\|_1-\|x\|_{(k)}=\tau\Big\}
  \]
  where $\|\cdot\|_{(k)}$ denotes the Ky Fan $k$-norm of vectors. Since
  $\Gamma^k$ is a polyhedral multifunction, i.e., its graph is the union of
  finitely many polyhedral convex sets, from \cite[Proposition 1]{Robinson81}
  it follows that each $\Gamma^k$ is calm at the origin for all $z\in\Gamma^k(0)$,
  which is equivalent to saying that for each $k\in\{1,\ldots,p\}$,
  there exist $\delta_k>0$ and $\gamma_k>0$ such that
  \begin{equation}\label{err-bound}
   {\rm dist}(z,\Gamma^k(0))\le\gamma_k\big[\|z\|_1-\|z\|_{(k)}\big]
   \quad\ \forall z\in\mathbb{B}(x^*,\delta_k).
  \end{equation}
  Set $\delta=\!\min\{\delta_1,\ldots,\delta_k,1/2\}$ and
  $\gamma=\!\max\{\gamma_1,\ldots,\gamma_k\}$. Fix an arbitrary $\epsilon\ge 0$
  and an arbitrary $(x,w)\in\mathbb{B}((x^*,w^*),\delta)\cap\mathcal{S}_{\epsilon}$.
  Take an arbitrary $\rho\ge\frac{\gamma\phi_{-}'(1)(1-t^*)L_{F_{\mu}}}{\nu(1-t_0)}$,
  where $t_0\in[0,1)$ is such that $\frac{1}{1-t^*}\in\partial\phi(t_0)$ and
  its existence is shown in \cite[Lemma 1]{LiuBiPan18}. Write
  \[
    J:=\Big\{j\in\{1,2,\ldots,p\}\ |\ \rho|x|_j^{\downarrow}>\phi_{-}'(1)\Big\}
    \ \ {\rm and}\ \ r=|J|.
  \]
  Clearly, $x\in\mathbb{B}(x^*,\delta_r)$. By invoking \eqref{err-bound}
  with $z=x$, there exists $\overline{x}^{\rho}\in\Gamma^r(0)$ such that
  \begin{equation}\label{err-ineq1}
    \|x-\overline{x}^{\rho}\|\le\gamma\big[\|x\|_1-\|x\|_{(r)}\big]
    =\gamma{\textstyle\sum_{j=r+1}^p}|x|_j^{\downarrow}.
  \end{equation}
  Let $J_1:=\big\{j\ |\ \frac{1}{1-t^*}\le\rho|x|_j^{\downarrow}\le\phi_{-}'(1)\big\}$
  and $J_2:=\big\{j\ |\ 0\le\rho|x|_j^{\downarrow}<\frac{1}{1-t^*}\big\}$.
  Notice that
  \[
   \sum_{i=1}^{p}\phi(w_i^{\downarrow})+\rho\big(\|x\|_1\!-\langle w,|x|\rangle\big)
   \ge\sum_{i=1}^{p}\min_{t\in[0,1]}\big\{\phi(t)+\rho|x|_i^{\downarrow}(1-t)\big\}.
  \]
  By invoking \cite[Lemma 1]{LiuBiPan18} with $\omega=|x|_j^{\downarrow}$
  for each $j$, it immediately follows that
  \begin{align*}
   &{\textstyle\sum_{i=1}^p}\phi(w_i^{\downarrow})+\rho\big(\|x\|_1\!-\langle w,x\rangle\big)\\
   &\ge\|\overline{x}^{\rho}\|_0+\frac{\rho(1-t_0)}{\phi_{-}'(1)(1\!-t^*)}
     \sum_{j\in J_1}\,|x|_j^{\downarrow}+\rho(1-t_0)
     \sum_{j\in J_2}\,|x|_j^{\downarrow}\nonumber\\
   &\ge\|\overline{x}^{\rho}\|_0+\frac{\rho(1-t_0)}{\phi_{-}'(1)(1\!-t^*)}
     \sum_{j\in J_1\cup J_2}|x|_j^{\downarrow}
   =\|\overline{x}^{\rho}\|_0+\frac{\rho(1-t_0)}{\phi_{-}'(1)(1\!-t^*)}
     \sum_{j=r+1}^p|x|_j^{\downarrow}\\
   &\ge \|\overline{x}^{\rho}\|_0+ \frac{\rho(1-t_0)}{\gamma\phi_{-}'(1)(1\!-t^*)}\|x-\overline{x}^{\rho}\|
   \ge \|\overline{x}^{\rho}\|_0+ \nu L_{F_{\mu}}\|x-\overline{x}^{\rho}\|
  \end{align*}
  where the second inequality is since $\phi(t^*)-\phi(1)\ge\phi_{-}'(1)(t^*\!-1)$,
  the third one is due to the inequality \eqref{err-ineq1},
  the last is using $\rho\ge\frac{\gamma\phi_{-}'(1)(1-t^*)L_{F_{\mu}}}{\nu(1-t_0)}$,
  and the equality is due to $J_1\cup J_2=\{1,\ldots,p\}\backslash J$ and $|J|=r$.
  Notice that $x\in\mathbb{B}_R$ since $\bigcup_{x^*\in\mathcal{S}^*}\mathbb{B}(x^*,1/2)
  \subseteq\mathbb{B}_{R}$ and $\overline{x}^{\rho}\in\mathbb{B}_R$
  by $\overline{x}^{\rho}\in\Gamma^{r}(0)$.
  From the Lipschitz continuity of $F_{\mu}$ over $\mathbb{B}_R$,
  it immediately follows that
  \(
    L_{F_{\mu}}\|x-\overline{x}^{\rho}\|
    \ge F_{\mu}(\overline{x}^{\rho})-F_{\mu}(x).
  \)
  Combining with the last inequality, we obtain
  \begin{equation}\label{temp-wxineq}
    {\textstyle\sum_{i=1}^p}\phi(w_i^{\downarrow})+\rho\big(\|x\|_1\!-\langle w,x\rangle\big)
    \ge \|\overline{x}^{\rho}\|_0+\nu^{-1}\big[F_{\mu}(\overline{x}^{\rho})-F_{\mu}(x)\big].
  \end{equation}
  Now take $\overline{w}_{i}^{\rho}=1$ for $i\in{\rm supp}(\overline{x}^{\rho})$
  and $\overline{w}_{i}^{\rho}=0$ for $i\notin{\rm supp}(\overline{x}^{\rho})$.
  Clearly, $(\overline{x}^{\rho},\overline{w}^{\rho})$ is a feasible point
  of the MPEC \eqref{Eprob} with $\sum_{i=1}^p\phi(\overline{w}_i^{\rho})
  =\|\overline{x}^{\rho}\|_0$. Then, it holds that
  \[
   F_{\mu}(\overline{x}^{\rho})+\nu\|\overline{x}^{\rho}\|_0
  \ge F_{\mu}(x^*) +\nu{\textstyle\sum_{i=1}^p}\phi(w_i^*).
  \]
  Together with the inequality \eqref{temp-wxineq},
  the stated inequality \eqref{aim-ineq} holds.
  By the arbitrariness of $(x^*,w^*)$ in $\mathcal{S}^*$, we obtain the desired
  result. The proof is then completed.
 \end{proof}

 Notice that Proposition \ref{partial-calm} still holds when $f$ is replaced by
 a general Lipschitzian function. Since the objective function of \eqref{Eprob}
 is coercive relative to $\Omega\times[0,e]$,
 combining Proposition \ref{partial-calm} with \cite[Proposition 2.1(b)]{LiuBiPan18},
 we have the following conclusion.
%-------------------------------------------------------------------------------------
 \begin{theorem}\label{theorem-epenalty}
  There exists a threshold $\overline{\rho}>0$ such that the following penalty problem
  \begin{equation}\label{Eprob-penalty}
  \min_{x\in\Omega,w\in[0,e]}\Big\{F_{\mu}(x)
     +\nu\big[\textstyle{\sum_{i=1}^p}\phi(w_i)+\rho\langle e\!-\!w,|x|\rangle\big]\Big\}
 \end{equation}
 associated to each $\rho\ge\overline{\rho}$ has the same global
 optimal solution set as the MPEC \eqref{Eprob} does. Also,
 this conclusion holds when $f$ is replaced by a lower bounded
 Lipschitzian function.
 \end{theorem}
 \begin{remark}\label{remark-epenalty}
  By the proof of Proposition \ref{partial-calm}, it follows that
  $\overline{\rho}=\frac{\gamma\phi_{-}'(1)(1-t^*)L_{F_{\mu}}}{\nu(1-t_0)}$,
  which depends on the Lipschitz constant of $F_{\mu}$ over the set
  $\mathbb{B}_{R}$, determined by $\mathcal{S}^*$, and the calmness constant
  $\gamma_k$ of $\Gamma^k$. So, it is generally hard to achieve an exact
  estimation on $\overline{\rho}$, and a varying $\rho$ is suggested in
  practical computation. We see that, as the regularization parameter $\nu$ increases,
  the corresponding threshold $\overline{\rho}$ becomes smaller,
  and consequently, it is easier to choose an appropriate $\rho$
  such that \eqref{Eprob-penalty} is a global exact penalty.
 \end{remark}

  It is well known that the handling of nonconvex constraints is
  much harder than the handling of nonconvex objective functions.
  Thus, Theorem \ref{theorem-epenalty} provides a convenient way
  to deal with the difficult MPEC \eqref{Eprob} and then the zero-norm
  regularized problem \eqref{prob}. In particular, the nonconvexity of
  the objective function of \eqref{Eprob-penalty} is owing to the coupled
  term $\langle e-w,|x|\rangle$ rather than the combination. Such a structure
  ensures that the problem \eqref{Eprob-penalty} associated to every
  $\rho>\overline{\rho}$ is an equivalent DC surrogate of \eqref{prob}.
  Indeed, by letting
  \begin{equation*}
   \psi(t):=\!\left\{\!\begin{array}{cl}
               \phi(t) &\textrm{if}\ t\in [0,1];\\
               +\infty & \textrm{otherwise}
              \end{array}\right.
  \end{equation*}
  and using the conjugate $\psi^*$ of $\psi$, the problem \eqref{Eprob-penalty}
  can be compactly written as follows:
  \begin{equation}\label{Eprob-surrogate}
   \min_{x\in\mathbb{R}^p}\Big\{\Theta_{\lambda,\rho}(x)
   \!:=F_{\mu}(x)+\delta_{\Omega}(x)+\lambda\|x\|_1
   -\lambda\rho^{-1}\textstyle{\sum_{i=1}^p}\psi^*(\rho|x_i|)\Big\}
   \ \ {\rm for}\ \lambda=\rho\nu.
  \end{equation}
  Since $\psi^*$ is a nondecreasing finite convex function in $\mathbb{R}$,
  clearly, $\Theta_{\lambda,\rho}$ is a DC function and the problem
  \eqref{Eprob-surrogate} associated to every $\rho\ge\overline{\rho}$
  is an equivalent DC surrogate of \eqref{prob}.

  \medskip

  When $\phi(t)=t$ for $t\in\mathbb{R}$, clearly, $\phi\in\mathscr{L}$ with $t^*=0$.
  An elementary calculation yields that $\psi^*(s)=\max(s-1,0)$ for $s\in\mathbb{R}$.
  Now the last two terms $\|x\|_1-\rho^{-1}\sum_{i=1}^p\psi^*(\rho|x_i|)$ in
  $\Theta_{\lambda,\rho}(x)$ is precisely the continuous relaxation $\Phi(x)$
  of the zero-norm given by \cite{Bian19}. For other choice of $\phi$, please refer to
  \cite[Appendix B]{LiuBiPan18}. In the rest of this work, we focus on
  \begin{equation}\label{phi}
   \phi(t):=\frac{a-1}{a+1}t^2+\frac{2}{a+1}t~~(a>1)~~{\rm for} ~~t\in \mathbb{R},
  \end{equation}
  for which an elementary calculation yields that the conjugate $\psi^*$
  has the following form
  \begin{equation}\label{psi-star}
   \psi^*(s)=\left\{\begin{array}{cl}
                      0 & \textrm{if}\ s\leq \frac{2}{a+1},\\
                      \frac{((a+1)s-2)^2}{4(a^2-1)} & \textrm{if}\ \frac{2}{a+1}<s\leq \frac{2a}{a+1},\\
                      s-1 & \textrm{if}\ s>\frac{2a}{a+1}.
                \end{array}\right.
  \end{equation}
 To close this section, we summarize the desirable properties of
 $\Theta_{\lambda,\rho}$ associated to this $\phi$. Their proofs are
 included in Appendix A where the following functions are needed:
 \begin{subequations}
 \begin{align}\label{grho}
  g_{\rho}(x)\!:=\rho^{-1}{\textstyle\sum_{i=1}^p}\varphi_{\rho}(x_i)
  \ \ {\rm with}\ \ \varphi_{\rho}(t):=\psi^*(\rho|t|)\ \ {\rm for}\ t\in\mathbb{R},\\
  \label{wrho}
  w_{\rho}(x)\!:=((\psi^*)'(\rho|x_1|),\ldots,(\psi^*)'(\rho|x_p|))^{\mathbb{T}}
  \quad\ \forall x\in\mathbb{R}^p.\quad
 \end{align}
 \end{subequations}

  \vspace{-0.4cm}
 %---------------------------------------------------------------------------------
 \begin{proposition}\label{prop-Theta}
  For any given $\nu>0$ and $\rho>0$, the following results hold with $\lambda=\rho\nu$.
  \begin{itemize}
   \item[(i)] $\Theta_{\lambda,\rho}$ is a lower bounded,
               coercive, semiconvex piecewise linear-quadratic function.

   \item[(ii)] For any given $x\in\mathbb{R}^p$, the generalized subdifferentials
               of $\Theta_{\lambda,\rho}$ at $x$ take the form of
               \[
                \!\widetilde{\partial}\Theta_{\lambda,\rho}(x)=
                 \widehat{\partial}\Theta_{\lambda,\rho}(x)=\partial\Theta_{\lambda,\rho}(x)
                 =A^{\mathbb{T}}\partial\!f(Ax\!-b)+\mu x+\mathcal{N}_\Omega(x)
                 +\lambda\partial\|x\|_1-\lambda\nabla g_{\rho}(x).
               \]

  \item[(iii)] If $\overline{x}\in\mathbb{R}^p$ is a (limiting) critical point of $\Theta_{\lambda,\rho}$
               with $|\overline{x}_{\rm nz}|\ge\frac{2a}{\rho(a+1)}$,
               then $\overline{x}$ is also a regular critical point of the zero-norm regularized problem \eqref{prob}.

  \item[(iv)] The function $\Theta_{\lambda,\rho}$ is a KL function of exponent $1/2$.
  \end{itemize}
 \end{proposition}
%----------------------------------------------------------------------------------
 \section{Proximal MM method for the problem \eqref{Eprob-surrogate}}\label{sec4}

  By Proposition \ref{prop-Theta} (i) and the discussion in the last section,
  the zero-norm regularized composite problem \eqref{prob} is equivalent to
  a piecewise linear-quadratic minimization problem \eqref{Eprob-surrogate}
  associated to $\rho>\overline{\rho}$.
  Although the objective function of \eqref{Eprob-surrogate} consists of
  a smooth $g_{\rho}$ with Lipschitzian gradient and a proper
  convex function $F_{\mu}(\cdot)+\|\cdot\|_1+\delta_{\Omega}(\cdot)$,
  this composite convex function makes the common proximal gradient method
  inapplicable to it. By introducing an additional variable $z\!\in\mathbb{R}^n$,
  the problem \eqref{Eprob-surrogate} can be rewritten as
  \begin{equation}\label{fz-ADMM}
  \min_{x\in\mathbb{R}^p,z\in\mathbb{R}^n}\Big\{f(z)+\frac{1}{2}\mu\|x\|^2+\lambda\|x\|_1-\lambda g_{\rho}(x)
  \ \ {\rm s.t.}\ \ Ax-z=b,\ x\in\Omega\Big\}
  \end{equation}
  so that the ADMM can be applied to solving it, but to the best of our knowledge
  there is no convergence guarantee for the ADMM to such a nonconvex nonsmooth problem.
  It is worthwhile to point out that the results developed in \cite{WangYZ19}
  for the ADMM is not suitable for the problem \eqref{fz-ADMM}.
  Motivated by the specific structure of $\Theta_{\lambda,\rho}$ and
  the recent work by Tang et al. \cite{Tang19}, we in this section develop
  a tailored proximal MM method for \eqref{Eprob-surrogate}.

  \medskip

  Fix an arbitrary $x'\in\mathbb{R}^p$. For any $x\in\mathbb{R}^p$,
  the convexity and smoothness of $\psi^*$ implies
  \begin{equation}\label{ineq-psistar}
   {\textstyle\sum_{i=1}^p}\psi^*(\rho|x_i|)\ge{\textstyle\sum_{i=1}^p}\psi^*(\rho|x_i'|)
   +\langle w_{\rho}(x'),\rho|x|-\rho|x'|\rangle
  \end{equation}
  where $w_{\rho}\!:\mathbb{R}^p\to\mathbb{R}^p$ is the mapping defined in \eqref{wrho}.
  Along with the expression of $\Theta_{\lambda,\rho}$,
  \[
   \Theta_{\lambda,\rho}(x)\le\Xi_{\lambda,\rho}(x,x')
   \!:=\!F_{\mu}(x)+\delta_{\Omega}(x)+\lambda\|x\|_1-\lambda\langle w_{\rho}(x'),|x|\rangle
    +R_{\lambda,\rho}(x')
  \]
  where $R_{\lambda,\rho}(x')=\lambda\langle w_{\rho}(x'),|x'|\rangle
  -\lambda\rho^{-1}\!\sum_{i=1}^p\psi^*\big(\rho|x_i'|\big)$.
  Notice that $\Xi_{\lambda,\rho}(x',x')=\Theta_{\lambda,\rho}(x')$.
  This means that $\Xi_{\lambda,\rho}(\cdot,x')$ is a majorization of
  $\Theta_{\lambda,\rho}$ at $x'$. This majorization is tighter than
  the one in \cite{Tang19} obtained by the convexity of $g_{\rho}$.
  Indeed, by Lemma \ref{alemma1} in Appendix,
  \[
   \langle w_{\rho}(x'),|x|-|x'|\rangle
   =\langle w_{\rho}(x'),|x|\rangle-\langle\nabla g_{\rho}(x'),x'\rangle
   \ge \langle\nabla g_{\rho}(x'),x\rangle-\langle \nabla g_{\rho}(x'),x'\rangle
  \]
  where the inequality is using $[w_{\rho}(x')]_i\ge 0$ for each $i$,
  and consequently for any $x\in\mathbb{R}^p$,
  \[
    \Xi_{\lambda,\rho}(x,x')\le F_{\mu}(x)+\delta_{\Omega}(x)+\lambda\|x\|_1-\lambda\langle\nabla g_{\rho}(x'),x\rangle
    +R_{\lambda,\rho}(x').
%    -\lambda\rho^{-1}\!\sum_{i=1}^p\psi^*\big(\rho|x_i'|\big)+\lambda\langle \nabla g_{\rho}(x'),x'\rangle.
  \]
  The majorization on the right hand side is precisely the one used by Tang et al. \cite{Tang19}.

  \medskip

  Our proximal MM method is designed by minimizing a proximal version of
  the majorization $\Xi(\cdot,x^{k})$ at the $k$th step, and its iterate steps
  are described as follows.
%-------------------------------------------------------------------------------------
 \begin{algorithm}[H]
 \caption{\label{Alg}{\bf(Proximal MM method for solving \eqref{Eprob-surrogate})}}
 \textbf{Initialization:}
  Choose a small $\mu>0$. Select $\lambda>0,\gamma_{1,0}>0,\gamma_{2,0}>0$,
  $\underline{\gamma_1}>0,\underline{\gamma_2}>0$ and $\varrho\in(0,1]$.
  Seek a starting point $x^0\in\Omega$ and a suitable $\rho\ge 1$. Set $k:=0$.

  \medskip
  \noindent
 \textbf{while} the stopping conditions are not satisfied \textbf{do}
 \begin{enumerate}

  \item  Let $w^{k}=w_{\rho}(x^{k})$. Compute the optimal solution $x^{k+1}$ of the following problem
         \begin{equation}\label{subprob-x}
           \min_{x\in\Omega}
            \Big\{F_{\mu}(x)+\lambda\langle e-w^{k},|x|\rangle
            +\frac{\gamma_{1,k}}{2}\|x-x^{k}\|^2+\frac{\gamma_{2,k}}{2}\|Ax-Ax^{k}\|^2\Big\}.
         \end{equation}

  \item Update the proximal parameters $\gamma_{1,k}$ and $\gamma_{2,k}$ by the following rule
        \[
          \gamma_{1,k+1}=\max(\underline{\gamma_1},\varrho \gamma_{1,k})\ \ {\rm and}\ \
          \gamma_{2,k+1}=\max(\underline{\gamma_2},\varrho \gamma_{1,k}).
        \]

  \item  Let $k\leftarrow k+1$, and go to Step 1.
  \end{enumerate}
  \textbf{end while}
 \end{algorithm}
%------------------------------------------------------------------------------
 \begin{remark}\label{remark-Alg}
 {\bf(i)} Although the problem \eqref{Eprob-surrogate} is a DC program,
  Algorithm \ref{Alg} does not belong to the DCA framework in \cite{LeThi14}
  since $w^k\ne\nabla g_{\rho}(x^{k})$ by Lemma \ref{alemma1} in Appendix A.
  In fact, by the definition of $w_{\rho}$ in \eqref{wrho} and the expression of $\psi^*$,
  it is easy to obtain that
  \begin{equation}\label{wik}
   w_i^{k}=\min\Big[1,\max\Big(0,\frac{(a\!+\!1)\rho|x_i^{k}|-2}{2(a\!-\!1)}\Big)\Big]
   \ \ {\rm for}\ i=1,\ldots,p.
  \end{equation}
  The proximal term $\frac{\gamma_{1,k}}{2}\|x-x^{k}\|^2+\frac{\gamma_{2,k}}{2}\|Ax-Ax^{k}\|^2$
  involved in the subproblems, inspired by the recent work \cite{Tang19}, plays a twofold role:
  one is to ensure that the subproblem \eqref{subprob-x} is solvable and the other,
  as will be shown in the sequel, is to guarantee the decreasing of
  the objective value sequence of the nonconvex problem \eqref{Eprob-surrogate}
  and then its global convergence.

  \medskip
  \noindent
  {\bf(ii)} It is well known that for nonconvex optimization problems,
  the choice of the starting point determines the quality of the limit
  of the sequence generated from this initial point. This means that
  the choice of $x^0$ in Algorithm \ref{Alg} is very crucial. Inspired by
  the good performance of $\ell_1$-norm regularized problem, we recommend to
  use the following $x^0$:
 \begin{equation}\label{x0sub}
  x^{0}\approx\!\mathop{\arg\min}_{x\in\mathbb{R}^p}\Big\{f(Ax\!-b)+\lambda\|x\|_1
  +\frac{\gamma_{1,0}}{2}\|x\|^2+\frac{\gamma_{2,0}}{2}\|Ax-b\|^2+\delta_{\Omega}(x)\Big\}.
 \end{equation}
  As will be shown in Subsection \ref{sec4.3}, such an $x^0$ is not bad at least
  in a statistical sense. In addition, we suggest that the parameter $\rho$
  is chosen by $\frac{c}{\|x^0\|_\infty}$ for a suitable $c>0$.
  By combining formula \eqref{wik} with the term $\langle e-w^{k},|x|\rangle$
  in \eqref{subprob-x}, such a choice of $\rho$ ensures that those very small
  $x_i^0$ (very likely corresponding to a zero entry) becomes zero quickly
  since a weight close to $1$ is imposed on $|x_i|$, those very large $x_i^0$
  (very likely corresponding to a nonzero entry) continue to be large
  since a weight closed to $0$ is imposed on $|x_i|$, and for the rest,
  a smaller $c$ means a larger weight imposed on them.

  \medskip
  \noindent
  {\bf(iii)} By the definition of $x^k$ and \cite[Theorem 23.8]{Roc70},
  it follows that for each $k\in\mathbb{N}$,
  \begin{align}\label{inclusion}
   0&\in\partial F_{\mu}(x^{k})+\mathcal{N}_{\Omega}(x^k)
    +\lambda\big[(1-w_1^{k-1})\partial|x_1^{k}|
     \times\cdots\times(1-w_p^{k-1})\partial|x_p^{k}|\big]\nonumber\\
   &\quad +\gamma_{1,k-1}(x^{k}-x^{k-1})+\gamma_{2,k-1}A^\mathbb{T}A(x^{k}-x^{k-1});
  \end{align}
  while by Proposition \ref{prop-Theta}(ii) and Lemma \ref{alemma1} in Appendix A,
  it holds that
  \begin{align*}
   \partial\Theta_{\lambda,\rho}(x^{k})
   &=\partial F_{\mu}(x^{k})+\mathcal{N}_{\Omega}(x^k)
     +\lambda\big[\partial|x_1^{k}|\times\cdots\times\partial|x_p^{k}|\big]
     -\lambda w^{k}\circ{\rm sign}(x^{k})\\
   &=\partial F_{\mu}(x^{k})+\mathcal{N}_{\Omega}(x^k)+\lambda\big[(1-w_1^{k})\partial|x_1^{k}|
     \times\cdots\times(1-w_p^{k})\partial|x_p^{k}|\big]
  \end{align*}
  where the second equality is by $w_i^{k}=0$ if $x_i^{k}=0$
  and $\partial|x_i^{k}|=\{{\rm sign}(x_i^{k})\}$ if $x_i^{k}\ne 0$. So,
  \begin{align*}
  &\lambda\big[(w_1^{k-1}\!-w_1^k)\partial|x_1^k|
     \times\cdots\times(w_p^{k-1}\!-w_p^k)\partial|x_p^k|\big]\\
  &+\gamma_{1,k-1}(x^{k-1}\!-x^k)+\gamma_{2,k-1}A^\mathbb{T}A(x^{k-1}\!-x^k)
   \in\partial\Theta_{\lambda,\rho}(x^{k}).
  \end{align*}
  Since each $\partial|x_i^k|\subseteq[-1,1]$,
  the following stopping criterion is suggested for Algorithm \ref{Alg}
  \begin{equation}\label{errk}
   {\bf Err}_k:=\frac{\|\lambda(w^{k-1}\!-w^k)+\big[\gamma_{1,k-1}I
   +\gamma_{2,k-1}A^\mathbb{T}A\big](x^{k-1}\!-x^k)\|}{1+\|b\|}\le{\rm tol}.
  \end{equation}
  This guarantees that the obtained $x^k$ is an approximate proximal critical point
  of $\Theta_{\lambda,\rho}$.
  \end{remark}
%--------------------------------------------------------------------------------
 \subsection{Convergence analysis of Algorithm \ref{Alg}}\label{sec4.1}

  We shall follow the recipe of the convergence analysis in \cite{Attouch13}
  for nonconvex nonsmooth optimization problems to establish the global
  and local linear convergence of Algorithm \ref{Alg}.
  As the first part of the recipe, we analyze the decreasing of the sequence $\{\Theta_{\lambda,\rho}(x^k)\}$.
%---------------------------------------------------------------------------------
 \begin{lemma}\label{descent}
  Let $\{x^k\}_{k\in\mathbb{N}}$ be the sequence given by Algorithm \ref{Alg}.
  Then, for each $k\ge 0$,
  \(
    \Theta_{\lambda,\rho}(x^{k})\ge\Theta_{\lambda,\rho}(x^{k+1})
    +\|x^{k+1}\!-x^{k}\|_{\gamma_{1,k}I+\gamma_{2,k}A^{\mathbb{T}}A}^2,
  \)
  and then $\sum_{k=0}^{\infty}\gamma_{1,k}\|x^{k+1}\!-x^{k}\|^2<\infty$.
 \end{lemma}
 \begin{proof}
  Fix an arbitrary $k\in\mathbb{N}\cup\{0\}$. By the definition of $x^k$, clearly, $x^k\in\Omega$.
  By invoking \eqref{ineq-psistar} with $x=x^{k+1}$ and $x'=x^{k}$
  and using the relation $w^k=w_{\rho}(x^k)$ for $k\ge 0$,
  \[
    \lambda\rho^{-1}\sum_{i=1}^p\psi^*(\rho|x_i^{k+1}|)-\lambda\big\langle w^{k},|x^{k+1}|\big\rangle
    \ge\lambda\rho^{-1}\sum_{i=1}^p\psi^*(\rho|x_i^{k}|)-\lambda\big\langle w^{k},|x^{k}|\big\rangle.
  \]
  Notice that the objective function of \eqref{subprob-x} is a sum of
  a convex function and a strongly convex quadratic function.
  Along with the definition of $x^{k+1}$ and $x^k,x^{k+1}\in\Omega$, we have
  \[
   F_{\mu}(x^{k})+\lambda\langle e\!-\!w^{k},|x^{k}|\rangle
   \ge F_{\mu}(x^{k+1})+\lambda\langle e\!-\!w^{k},|x^{k+1}|\rangle
        +\|x^{k+1}\!-\!x^{k}\|_{\gamma_{1,k}I+\gamma_{2,k}A^{\mathbb{T}}A}^2.
  \]
  By adding the last two inequalities, we immediately obtain that
  \begin{align*}
   & F_{\mu}(x^{k})+\lambda\|x^{k}\|_1-\lambda\rho^{-1}{\textstyle\sum_{i=1}^p}\psi^*(\rho|x_i^{k}|)\\
   &\ge F_{\mu}(x^{k+1})+\lambda\|x^{k+1}\|_1-\lambda\rho^{-1}{\textstyle\sum_{i=1}^p}\psi^*(\rho|x_i^{k+1}|)
      +\|x^{k+1}\!-\!x^{k}\|_{\gamma_{1,k}I+\gamma_{2,k}A^{\mathbb{T}}A}^2.
  \end{align*}
  This, by the definition of the function $\Theta_{\lambda,\rho}$,
  is equivalent to saying that
  \[
    \Theta_{\lambda,\rho}(x^{k})\ge\Theta_{\lambda,\rho}(x^{k+1})
    +\|x^{k+1}\!-\!x^{k}\|_{\gamma_{1,k}I+\gamma_{2,k}A^{\mathbb{T}}A}^2,
  \]
  which implies the first part of the conclusions.
  Notice that the sequence $\{\Theta_{\lambda,\rho}(x^k)\}_{k\in\mathbb{N}}$
  is nonincreasing. From its lower boundedness of $\Theta_{\lambda,\rho}$
  in Proposition \ref{prop-Theta} (i), the sequence $\{\Theta_{\lambda,\rho}(x^k)\}_{k\in\mathbb{N}}$
  is convergent. Combining this with the first part of the conclusions,
  we obtain $\sum_{k=0}^{\infty}\gamma_{1,k}\|x^{k+1}-x^{k}\|^2<\infty$.
  The proof is then completed.
 \end{proof}

 The following lemma provides a subgradient lower bound for the iterate gaps.
%-------------------------------------------------------------------------------------
 \begin{lemma}\label{Subgrad-bound}
 For each $k\in\mathbb{N}$, there exists a vector
 $\Delta u^{k}\in\partial \Theta_{\lambda,\rho}(x^{k})$
 such that
 \[
   \|\Delta u^{k}\|\le\Big[\gamma_{1,k-1}+\lambda\rho\max\Big(1,\frac{a\!+\!1}{2(a\!-\!1)}\Big)
   +\gamma_{2,k-1}\|A^{\mathbb{T}}A\|\Big]\|x^{k-1}-x^{k}\|.
 \]
 \end{lemma}
 \begin{proof}
  From the discussion in Remark \ref{remark-Alg} (iii), for each $k\in\mathbb{N}$ it holds that
  \[
   \lambda(w^{k-1}\!-\!w^{k})\circ{\rm sign}(x^{k})+\gamma_{1,k-1}(x^{k-1}\!-\!x^{k})
   +\gamma_{2,k-1}A^{\mathbb{T}}A(x^{k-1}\!-\!x^{k})\in\partial\Theta_{\lambda,\rho}(x^{k}).
  \]
  Take $\Delta u^k=\lambda(w^{k-1}\!-\!w^{k})\circ{\rm sign}(x^{k})+\gamma_{1,k-1}(x^{k-1}\!-\!x^{k})
   +\gamma_{2,k-1}A^{\mathbb{T}}A(x^{k-1}\!-\!x^{k})$.
  Since $(\psi^*)'$ is Lipschitzian of modulus $\max(1,\frac{a+1}{2(a-1)})$,
  from $w^k=w_{\rho}(x^k)$ and \eqref{wrho} we have
  \[
   \|(w^{k-1}\!-\!w^{k})\circ {\rm sign}(x^{k})\|\le\|w_{\rho}(x^{k-1})-w_{\rho}(x^{k})\|
   \le \rho\max\Big(1,\frac{a+1}{2(a\!-\!1)}\Big)\|x^{k-1}\!-\!x^{k}\|.
  \]
  The desired result follows from the last two equations. The proof is completed.
 \end{proof}

  By combining Lemma \ref{descent} and \ref{Subgrad-bound} with
  Proposition \ref{prop-Theta} (iv) and following the similar arguments
  as those for \cite[Theorem 3.2 \& 3.4]{Attouch10},
  we get the following convergence results.
%-----------------------------------------------------------------------------------
 \begin{theorem}\label{converge-linearly}
  Let $\{x^k\}_{k\in\mathbb{N}}$ be generated by Algorithm \ref{Alg}.
  The following statements hold.
  \begin{itemize}
   \item[(i)] The sequence $\{x^k\}_{k\in\mathbb{N}}$ has a finite length, i.e.,
              $\sum_{k=1}^{\infty}\|x^{k+1}-x^k\| <\infty$.

   \item[(ii)] The sequence $\{x^k\}_{k\in\mathbb{N}}$ converges to a critical point
               of $\Theta_{\lambda,\rho}$, say $\overline{x}$, which is also a regular
               critical point of the zero-norm regularized problem \eqref{prob}
               if $|\overline{x}_{\rm nz}|\ge \frac{2a}{\rho(a+1)}$.

   \item[(ii)] The sequence $\{x^k\}_{k\in\mathbb{N}}$ converges to $\overline{x}$ in a Q-linear rate.
  \end{itemize}
 \end{theorem}
%--------------------------------------------------------------------------------
 \subsection{Local optimality of critical points}\label{sec4.2}

  We have established that the sequence $\{x^k\}_{k\in \mathbb{N}}$
  generated by Algorithm \ref{Alg} is convergent, and converges Q-linearly
  to a proximal critical point of $\Theta_{\lambda,\rho}$.
  It is natural to ask whether such a critical point is a local minimum of
  $\Theta_{\lambda,\rho}$ or not. If yes, is it locally optimal to
  the zero-norm problem \eqref{prob}?
  Next we provide affirmative answers to the two questions.
%------------------------------------------------------------------------------
  \begin{theorem}\label{optimality}
   Fix arbitrary $\lambda>0$ and $\rho>0$. Consider an arbitrary $\overline{x}$
   with $0\in\partial\Theta_{\lambda,\rho}(\overline{x})$ and
   $|\overline{x}_{\rm nz}|>\frac{2a}{\rho(a+1)}$.
   Then, for any $0\ne\zeta\in\mathbb{R}^p$, it holds that
   $d^2\Theta_{\lambda,\rho}(\overline{x}|0)(\zeta)>0$, and consequently,
   $\overline{x}$ is a local optimal solution of the problem
   \eqref{Eprob-surrogate}.
  \end{theorem}
  \begin{proof}
   Fix an arbitrary nonzero $\zeta$.
   Define $G_{\lambda,\rho}(z)\!:=\lambda\big[\|z\|_1-g_{\rho}(z)\big]$ for $z\in\mathbb{R}^p$.
   Clearly, $G_{\lambda,\rho}$ is a locally Lipschitz and regular function.
   Let $\widetilde{F}_{\mu}(z)\!:=F_{\mu}(z)+\delta_{\Omega}(z)$ for $z\in\mathbb{R}^p$.
   Notice that $\Theta_{\lambda,\rho}(z)=\widetilde{F}_{\mu}(z)+G_{\lambda,\rho}(z)$
   for $z\in\mathbb{R}^p$. By \cite[Proposition 13.19]{RW98}, we have
   \begin{equation}\label{SOD-Theta}
    d^2\Theta_{\lambda,\rho}(\overline{x}|0)(\zeta)
     \ge\sup_{u\in\partial\widetilde{F}_{\mu}(\overline{x})\atop\xi\in\partial G_{\lambda,\rho}(\overline{x})}
      \Big\{ d^2\widetilde{F}_{\mu}(\overline{x}\,|\,u)(\zeta)+ d^2G_{\lambda,\rho}(\overline{x}|\xi)(\zeta)
      \ \ {\rm s.t.}\ \ u+\xi=0\Big\}.
  \end{equation}
  Recall that $\widetilde{F}_{\mu}(z)\!=F(z)+\delta_{\Omega}(z)+\frac{1}{2}\mu\|z\|^2$
  with $F(z):=f(Az-b)$ for $z\in\mathbb{R}^p$. By \cite[Proposition 13.9]{RW98},
  $F+\delta_{\Omega}$ is twice epi-differentiable at $\overline{x}$ for each
  $u'\in\partial(F+\delta_{\Omega})(\overline{x})$. Moreover,
  from \cite[Exercise 13.18]{RW98}, it follows that
  \begin{equation}\label{SOD-Fmu}
    d^2\widetilde{F}_{\mu}(\overline{x}\,|\,u)(\zeta)
    =d^2(F+\delta_{\Omega})(\overline{x}\,|\,u-\!\mu\overline{x})(\zeta)+\mu\|\zeta\|^2>0
    \quad\forall u\in\partial\widetilde{F}_{\mu}(\overline{x})
  \end{equation}
  where the inequality is due to
  $d^2(F+\delta_{\Omega})(\overline{x}\,|\,u-\!\mu\overline{x})(\zeta)\ge 0$
  by \cite[Proposition 13.9]{RW98}. Let $h(z):=\|z\|_1$ for $z\in\mathbb{R}^p$.
  Fix an arbitrary $\xi\in\partial G_{\lambda,\rho}(\overline{x})
   =\lambda\big[\partial h(\overline{x})-\!\nabla g_{\rho}(\overline{x})\big]$.
   Since $G_{\lambda,\rho}$ is locally Lipschitz continuous and directionally differentiable,
   it holds that
   \[
    dG_{\lambda,\rho}(\overline{x})(\zeta)=G_{\lambda,\rho}'(\overline{x};\zeta)
     =\lambda h'(\overline{x};\zeta)-\lambda\langle\nabla g_{\rho}(\overline{x}),\zeta\rangle
     \ge\langle\xi,\zeta\rangle,
   \]
   where $dG_{\lambda,\rho}\!: \mathbb{R}^{p}\to[-\infty,+\infty]$ is the subderivative
   function of $G$. By \cite[Proposition 13.5]{RW98}, $d^2G_{\lambda,\rho}(\overline{x}|\xi)(\zeta)=+\infty$
   when $dG_{\lambda,\rho}(\overline{x})(\zeta)>\langle\xi,\zeta\rangle$,
   which along with \eqref{SOD-Theta}-\eqref{SOD-Fmu} implies that
   $d^2\Theta_{\lambda,\rho}(\overline{x}|0)(\zeta)>0$.
   So, it suffices to consider the case where
   $G_{\lambda,\rho}'(\overline{x};\zeta)=\langle\xi,\zeta\rangle$.
   In this case, by the definition of the second subderivative, it follows that
   \begin{align*}
    d^2G_{\lambda,\rho}(\overline{x}|\xi)(\zeta)
    &=\liminf_{\tau\downarrow 0,\,\zeta'\to\zeta}
     \frac{G_{\lambda,\rho}(\overline{x}+\tau\zeta')-G_{\lambda,\rho}(\overline{x})
     -\tau G_{\lambda,\rho}'(\overline{x};\zeta')}{\frac{1}{2}\tau^2}
      \nonumber\\
    &=\lambda\liminf_{\tau\downarrow 0,\,\zeta'\to\zeta}
      \frac{-g_{\rho}(\overline{x}+\tau\zeta')+g_{\rho}(\overline{x})
      +\tau \langle\nabla g_{\rho}(\overline{x}),\zeta'\rangle}{\frac{1}{2}\tau^2}
   \end{align*}
   where the second equality is due to
   $\|\overline{x}+\tau\zeta'\|_1-\|\overline{x}\|_1-\tau h'(\overline{x};\zeta')=0$
   for any sufficiently small $\tau>0$. Since $|\overline{x}_{\rm nz}|>\frac{2a}{\rho(a-1)}$,
   it is immediate to have that $\{1,2,\ldots,p\}=I\cup J$ with
   \[
     I:=\Big\{i\ |\ |\overline{x}_i|<\frac{2}{\rho(a\!+\!1)}\Big\}\ \ {\rm and}\ \
     J:=\Big\{i\ |\ |\overline{x}_i|>\frac{2a}{\rho(a\!+\!1)}\Big\}.
   \]
   Recall that $g_{\rho}(z)=\rho^{-1}\sum_{i=1}^p\varphi_{\rho}(z_i)$
   for $z\in\mathbb{R}^p$ by \eqref{grho}. From the last two equations,
   \[
    d^2G_{\lambda,\rho}(\overline{x}|\xi)(\zeta)
    =\lambda\rho^{-1}\liminf_{\tau\downarrow 0,\,\zeta'\to\zeta}
      \frac{\sum_{i\in I\cup J}\big[-\varphi_{\rho}(\overline{x}_i+\tau\zeta_i')+\varphi_{\rho}(\overline{x}_i)
      +\tau\varphi_{\rho}'(\overline{x}_i)\zeta_i'\big]}{\frac{1}{2}\tau^2}.
   \]
   By the expression of $\varphi_{\rho}$ and the formula of
   $\varphi_{\rho}'$ in \eqref{der-varphi-rho}, for all sufficiently
   small $\tau>0$ and all $\zeta'$ sufficiently close to $\zeta$, we have
   \(
    -\varphi_{\rho}(\overline{x}_i+\tau\zeta_i')+\varphi_{\rho}(\overline{x}_i)
      +\tau\varphi_{\rho}'(\overline{x}_i)\zeta_i'=0
   \)
 for $i\in I\cup J$. Along with the last equation and \eqref{SOD-Theta}-\eqref{SOD-Fmu},
 $d^2\Theta_{\lambda,\rho}(\overline{x}|0)(\zeta)>0$.
  By the arbitrariness of $\zeta\in\mathbb{R}^p\backslash\{0\}$
  and \cite[Theorem 13.24]{RW98}, $\overline{x}$ is a local optimal
  solution of \eqref{Eprob-surrogate}.
 \end{proof}

  Theorem \ref{optimality} states that those critical points
  of $\Theta_{\lambda,\rho}$ with not too small nonzero entries
  are local minima of $\theta_{\lambda,\rho}$. To answer the second question,
  we need the following lemma.
%-------------------------------------------------------------------------------------
  \begin{lemma}\label{znorm-optimal}
   Fix an arbitrary $\nu>0$. Define $\Theta(z)\!:=F_{\mu}(z)+\delta_{\Omega}(z)+\nu\|z\|_0$
   for $z\in\mathbb{R}^p$. Consider an arbitrary $\overline{x}$ with $0\in\partial\Theta(\overline{x})$.
   If $v^1+v^2=0$ for $v^1\in\mathcal{N}_{\Omega}(\overline{x})$
   and $v^2\in\partial\|\cdot\|_0(\overline{x})$ implies $(v^1,v^2)=(0,0)$,
   then $d^2\Theta(\overline{x}|0)(\zeta)>0$ for any $0\ne\zeta\in\mathbb{R}^p$,
   and consequently, $\overline{x}$ is a local optimal solution of the problem \eqref{prob}.
  \end{lemma}
  \begin{proof}
   Fix an arbitrary $0\ne\zeta\in\mathbb{R}^p$. Let $h(z):=\nu\|z\|_0$ for $z\in\mathbb{R}^p$
   and $\widetilde{F}_{\mu}$ be same as in the proof of Theorem \ref{optimality}.
   From the given assumption and \cite[Proposition 13.19]{RW98},
  \begin{equation}\label{SOD-Theta1}
    d^2\Theta(\overline{x}|0)(\zeta)
    \ge\sup_{u\in\partial\widetilde{F}_{\mu}(\overline{x}),\,v\in\partial\|\cdot\|_0(\overline{x})}
      \!\Big\{ d^2\widetilde{F}_{\mu}(\overline{x}\,|\,u)(\zeta)+ d^2h(\overline{x}|v)(\zeta)
    \ \ {\rm s.t.}\ \ u+v=0\Big\}.
   \end{equation}
  Write $\overline{J}:=\{1,2,\ldots,p\}\backslash{\rm supp}(\overline{x})$.
  Fix an arbitrary $v\in\partial h(\overline{x})=[\![\overline{x}]\!]^{\perp}$
  where the equality is due to \cite{Le13}. Then
  $\langle v,\zeta\rangle=\langle v_{\overline{J}},\zeta_{\overline{J}}\rangle$.
  On the other hand, a simple calculation yields
  \[
    dh(\overline{x})(\zeta)={\textstyle\sum_{i\in\overline{J}}}\,\delta_{\{0\}}(\zeta_i).
  \]
  This means that $dh(\overline{x})(\zeta)\ge\langle v,\zeta\rangle$.
  By \cite[Proposition 13.5]{RW98}, $d^2h(\overline{x}|\xi)(\zeta)=+\infty$
  when $dh(\overline{x})(\zeta)>\langle v,\zeta\rangle$,
  which along with \eqref{SOD-Fmu} and \eqref{SOD-Theta1} implies that
   $d^2\Theta(\overline{x}|0)(\zeta)>0$. So, it suffices to consider
  the case where $dh(\overline{x})(\zeta)=\langle v_{\overline{J}},\zeta_{\overline{J}}\rangle$.
  In this case, from the last equation, $\zeta_{\overline{J}}=0$.
  By the definition of the second derivative, $d^2h(\overline{x}|v)(\zeta)$ equals
  \[
    \liminf_{\tau\downarrow 0,\,\zeta'\to\zeta}
    \frac{h(\overline{x}+\!\tau\zeta')-h(\overline{x})-\tau\langle v_{\overline{J}},\zeta_{\overline{J}}'\rangle}{\frac{1}{2}\tau^2}
    \!=\!\liminf_{\tau\downarrow 0,\,\zeta_{\overline{J}}'\to\zeta_{\overline{J}}}\!
    \frac{\sum_{i\in\overline{J}}\big[{\rm sign}(\tau|\zeta_i'|)-\!\tau v_i\zeta_i'\big]}{\frac{1}{2}\tau^2}\ge 0.
  \]
  This along with \eqref{SOD-Fmu} and \eqref{SOD-Theta1} implies that
  $d^2\Theta(\overline{x}|0)(\zeta)>0$. By the arbitrariness of $\zeta\in\mathbb{R}^p\backslash\{0\}$
  and \cite[Theorem 13.24]{RW98}, $\overline{x}$ is a local optimum of \eqref{prob}.
 \end{proof}

 By combining Theorem \ref{optimality} and Lemma \ref{znorm-optimal},
 we obtain the following conclusion.
%--------------------------------------------------------------------------
 \begin{corollary}\label{corollary-convergence}
  Let $\{x^k\}_{k\in\mathbb{N}}$ be the sequence generated by Algorithm \ref{Alg}.
  If its limit $\overline{x}$ satisfies $|\overline{x}_{\rm nz}|>\frac{2a}{\rho(a+1)}$,
  then $\overline{x}$ is a local minimum of the problem \eqref{Eprob-surrogate}.
  If in addition the constraint qualification in Lemma \ref{znorm-optimal} holds,
  then $\overline{x}$ is also a local minimum of \eqref{prob}.
 \end{corollary}
%----------------------------------------------------------------------------------
 \subsection{Statistical error bound of the limit $\overline{x}$}\label{sec4.3}

 In this part we focus on the scenario where $b\in\mathbb{R}^p$
 is from the linear observation model
 \begin{equation}\label{observe}
   b_i= a_i^{\mathbb{T}}x^*+\varpi_i,\ \ i=1,2,\ldots,n
 \end{equation}
 where each row $a_i^{\mathbb{T}}$ of $A$ follows the normal distribution
 $N(0,\Sigma)$, $x^*\in\Omega$ is the true but unknown sparse vector
 with sparsity $s^*\ll p$, and $\varpi_i\in\mathbb{R}$ is the noise.
 We assume that $\varpi=(\varpi_1,\ldots,\varpi_n)^{\mathbb{T}}$
 is nonzero and $f(z)=\frac{1}{n}\sum_{i=1}^n\theta(z_i)$
 with $\theta$ satisfying Assumption \ref{theta-assump}.
%----------------------------------------------------------------------------
 \begin{assumption}\label{theta-assump}
  The function $\theta\!:\mathbb{R}\to\mathbb{R}_{+}$ is convex with
  $\theta(0)=0$ and its square $\theta^2$ is strongly convex with modulus
  $\tau>0$. By \cite[Theorem 23.4]{Roc70}, there exists $\widetilde{\tau}>0$ such that
 \begin{equation}\label{subdiff-theta}
  |\eta|\le\widetilde{\tau}\quad{\rm for\ any}\ \eta\in\partial\theta(t)\ {\rm and\ any}\ t\in\mathbb{R}.
 \end{equation}
 \end{assumption}

 Since we are interested in the high-dimensional case, i.e. $n<p$,
 the sample covariance matrix $\frac{1}{n}A^{\mathbb{T}}A$ is not positive definite,
 but it may be positive definite on a subset $\mathcal{C}$ of $\mathbb{R}^p$.
 Specifically, for a given subset $S\subset\{1,\ldots,p\}$, we say that the sample
 covariance matrix $\frac{1}{n}A^{\mathbb{T}}A$ satisfies the restricted eigenvalue
 (RE) condition over $S$ with parameter $\kappa>0$ if
 \[
   \frac{1}{2n}\|Ax\|^2\ge\kappa\|x\|^2\quad{\rm for\ all}\
    x\in\mathbb{R}^p\ \ {\rm with}\ \|x_{S^c}\|_1\le 3\|x_{S}\|_1.
 \]
 In the sequel, we need a RE condition of $\frac{1}{n}A^{\mathbb{T}}A$
 over $\mathcal{C}(S^*)$ with parameter $\kappa>0$ where
 \[
   \mathcal{C}(S^*)\!:=\bigcup_{S^{*}\subset S,|S|\le 1.5s^*}\!\Big\{x\in\mathbb{R}^p\!:
   \|x_{S^c}\|_1\le 3\|x_{S}\|_1\Big\}.
 \]
 Obviously, $\mathcal{C}(S^*)$ comprises those vector $x$ whose components $x_j$
 for $j\notin S^*$ are small. By \cite[Corollary 1]{Raskutti10}, when $\Sigma$
 satisfies the RE condition over $\mathcal{C}(S^*)$ with parameter $\kappa>0$
 (for example, $\Sigma$ is positive definite),
 for $n>c''\frac{\max_{1\le j\le p}\Sigma_{jj}}{\kappa}s^*\log p$,
 the matrix $\frac{1}{n}A^{\mathbb{T}}A$ satisfies the RE condition over
 $\mathcal{C}(S^*)$ of parameter $\sqrt{2\kappa}/8$ with probability at least
 $1-c'\exp(-cn)$, where $c,c'$ and $c''$ are the universal positive constants.
 This means that for a larger $n$, $\frac{1}{n}A^{\mathbb{T}}A$ has the RE
 property over $\mathcal{C}(S^*)$ of a large $\kappa>0$ with a high probability.

 \medskip

 By Lemma \ref{xk-lemma2} in Appendix B, we can establish the following
 error bound to the true $x^*$ for the limit $\overline{x}$ of $\{x^k\}_{k\in\mathbb{N}}$
 under the RE condition of $\frac{1}{n}A^{\mathbb{T}}A$ over the set $\mathcal{C}(S^*)$.
%------------------------------------------------------------------------------
 \begin{theorem}\label{error-bound1}
  Suppose that Assumption \ref{theta-assump} holds and $\frac{1}{n}A^{\mathbb{T}}A$
  satisfies the RE condition of parameter $\kappa>0$ over $\mathcal{C}(S^*)$.
  Write $\mathcal{I}:=\big\{i\in\{1,\ldots,n\}\ |\ \varpi_i\ne 0\big\}$.
  Take a constant $c>\frac{1}{\mu\widetilde{\tau}\|\varpi\|_\infty+0.5\tau\kappa
  -\widetilde{\tau}\|A\|_{\infty}(9n^{-1}\!\interleave\!A_{\mathcal{I}\cdot}\!\interleave_1
  +9.5\mu\|x^*\|_\infty)\sqrt{6s^*}}$. If the parameters $\lambda$ and $\rho$ are chosen such that
  \(
  \lambda\in\!\Big[\frac{8}{n}\!\interleave\!A_{\mathcal{I}\cdot}\!\interleave_1+8\mu\|x^*\|_\infty,
  \frac{2\mu\widetilde{\tau}\|\varpi\|_\infty+\tau\kappa-2c^{-1}
  -\widetilde{\tau}\|A\|_{\infty}(2n^{-1}\!\interleave\!A_{\mathcal{I}\cdot}\!\interleave_1
  +3\mu\|x^*\|_\infty)\sqrt{6s^*}}{2\widetilde{\tau}\|A\|_{\infty}\sqrt{6s^*}}\Big]
  \)
  and $1\le\rho\le\frac{8}{9.5\sqrt{3}c\widetilde{\tau}\lambda\|\varpi\|_\infty}$,
  then for the limit $\overline{x}$ of $\{x^k\}$ with $|\overline{x}_i|\le\frac{a}{\rho(a+1)}$
  for $i\in (S^*)^c$,
  \begin{equation}\label{error-ineq1}
    \|\overline{x}- x^*\|\le\frac{9c\widetilde{\tau}\lambda\sqrt{1.5s^*}}{8}\|\varpi\|_\infty.
  \end{equation}
 \end{theorem}
 \begin{remark}\label{remark-errbound}
 {\bf(i)} Theorem \ref{error-bound1} states that for the problem \eqref{prob}
 with $(b,A)$ from the model \eqref{observe} and $f(z)=\frac{1}{n}\sum_{i=1}^n\theta(z_i)$
 for $\theta$ satisfying Assumption \ref{theta-assump},
 when applying Algorithm \ref{Alg} with appropriate $\lambda$ and $\rho$
 for solving its surrogate problem \eqref{Eprob-surrogate},
 the limit $\overline{x}$ of the generated sequence satisfies \eqref{error-ineq1}
 provided that $|\overline{x}_i|\le\frac{a}{\rho(a+1)}$ for $i\in (S^*)^c$.
 Such a requirement on $\overline{x}$ is rather mild since it allows $\overline{x}$
 to have more small nonzero entries than $x^*$, which seems more reasonable than
 the one used in \cite[Assumption 3.7]{Cao18} for smooth loss functions.

 \medskip
 \noindent
 {\bf(ii)} Recall that $a_1^{\mathbb{T}},\ldots, a_n^{\mathbb{T}}$ are i.i.d. and follow
 $N(0,\Sigma)$. All entries of each column of the submatrix $A_{\mathcal{I}\cdot}$
 are independent and follow normal distribution. By \cite[Lemma 5.5]{Vershynin12},
 it follows that $\!\interleave A_{\mathcal{I}\cdot}\interleave_1=O(|\mathcal{I}|)$.
 Thus, when the noise vector $\varpi$ is sparse enough, say, $|\mathcal{I}|=O(\sqrt{n})$
 and $s^*=O(n^{\varsigma})$ with $\varsigma\in[0,1/2)$,
 by recalling that $\mu>0$ is a very small constant, we have
 $0.5\tau\kappa\gg\widetilde{\tau}\|A\|_{\infty}(9n^{-1}\!\interleave\!A_{\mathcal{I}\cdot}\!\interleave_1
  +9.5\mu\|x^*\|_\infty)\sqrt{6s^*}$ for an appropriately large $n$,
 and consequently, the choice interval of $\lambda$ is nonempty and
 will become larger as $n$ increases. Different from \cite{Loh15} for the smooth loss,
 our choice interval of $\lambda$ does not involve the noise but requires a certain restriction
 on the sparsity of $\varpi$. We see that as the sparsity of $\varpi$
 increases (i.e., $|\mathcal{I}|$ decreases), the value of $\lambda$ becomes smaller
 and then the error bound of $\overline{x}$ to the true $x^*$ becomes better.
 In addition, similar to the $\ell_1$-regularized squared-root loss in \cite{Belloni110},
 the parameter $\lambda$ is required to belong to an interval depending on
 the sparsity $s^*$. Clearly, a large $s^*$ implies a small interval of
 the parameter $\lambda$.
 \end{remark}

 Notice that the limit $\overline{x}$ depends on the starting point $x^0$.
 The following result states that, when $x^0$ is chosen by \eqref{x0sub},
 $\|x^0-x^*\|=C\sqrt{s^*}$ for a positive constant $C$ dependent on
 $\gamma_{1,0}$ and $\gamma_{2,0}$, which will become smaller for a larger
 $\gamma_{1,0}$ and a smaller $\gamma_{2,0}$.
%-----------------------------------------------------------------------------------
 \begin{proposition}\label{x0err-bound}
  Let $\Theta\!:\mathbb{R}^p\to(-\infty,+\infty]$ be the objective function
  of \eqref{x0sub} and $x^0$ be an approximate optimal solution in the sense
  that there exist $\xi^0$ and $\epsilon\ge 0$ with $\|\xi^0\|\le\epsilon$
  such that $\xi^0\in\!\partial\Theta(x^0)$. Suppose that Assumption \ref{theta-assump}
  holds. If the parameter $\lambda$ is chosen such that
  \(
   \lambda\ge 2(n^{-1}\widetilde{\tau}\!\interleave\!A\!\interleave_1+\gamma_{1,0}\| x^*\|_{\infty}
   +\gamma_{2,0}\|A^{\mathbb{T}}\varpi\|_{\infty}+\epsilon),
  \)
  then
  \(
    \| x^0-x^*\|\le\frac{3\lambda\sqrt{s^*}}{2\gamma_{1,0}}.
  \)
 \end{proposition}
%------------------------------------------------------------------------
 \section{Numerical experiments}\label{sec5}

 We test the performance of Algorithm \ref{Alg} by applying it
 to the problem \eqref{prob} with $\Omega=\mathbb{R}^p$ and
 $f(z)=\frac{1}{n}\sum_{i=1}^n\theta(z_i)$ for $\theta(t)=|t|$.
 By Theorem \ref{converge-linearly} and Corollary \ref{corollary-convergence},
 the sequence $\{x^k\}$ generated by Algorithm \ref{Alg} converges Q-linearly to
 a limit $\overline{x}$ which is also a local minimum of \eqref{Eprob-surrogate}
 if $|\overline{x}_{nz}|\ge\frac{2a}{\rho(a+1)}$. Since $\theta$ satisfies
 Assumption \ref{theta-assump}, by Theorem \ref{error-bound1} and
 Remark \ref{remark-errbound} (ii), there is a high probability for
 $\overline{x}$ to be close to the true $x^*$ if the noise $\varpi$
 is sparse enough. To validate the efficiency of Algorithm \ref{Alg} via numerical comparison,
 we here provide a globally convergent ADMM for the partially smoothed form of \eqref{Eprob-surrogate}.
%------------------------------------------------------------------------
 \subsection{iPADMM for partially smoothed surrogate}\label{sec5.2}

  As mentioned in the beginning of Section \ref{sec4}, when the ADMM is
  directly applied to the surrogate problem \eqref{Eprob-surrogate} or
  its equivalent problem \eqref{fz-ADMM}, there is no convergence certificate.
  Recall that the Moreau envelope $e_{\varepsilon}f$ of $f$ with the parameter
  $\varepsilon>0$ is smooth and $\nabla e_{\varepsilon}f$ is globally Lipschitz
  continuous. Moreover, $f(z)-\frac{\varepsilon}{2n}\le e_{\varepsilon}f(z)\le f(z)$
  by noting that
  \[
    e_{\varepsilon}f(z)=\sum_{i=1}^n e_{\varepsilon}(n^{-1}\theta)(z_i)\ \ {\rm with}\ \
    e_{\varepsilon}(n^{-1}\theta)(t):=\left\{\begin{array}{cl}
                               \!\frac{1}{n}|t|-\frac{\varepsilon}{2n^2}  &{\rm if}\ |t|>\frac{\varepsilon}{n};\\
                               {t^2}/{(2\varepsilon)}         &{\rm if}\ |t|\le\frac{\varepsilon}{n}.
                               \end{array}\right.
  \]
  We replace $f$ in \eqref{fz-ADMM} by its Moreau envelope $e_{\varepsilon}f$
  and apply the ADMM with an indefinite-proximal term (iPADMM for short) to
  the partially smoothed formulation of \eqref{fz-ADMM}:
  \begin{equation}\label{smooth-surrogate}
  \min_{x\in\mathbb{R}^p,z\in\mathbb{R}^n}\Big\{e_{\varepsilon}f(z)+({\mu}/{2})\|x\|^2+\vartheta_{\lambda,\rho}(x)
  \ \ {\rm s.t.}\ \ Ax-b-z=0\Big\},
 \end{equation}
 where, for any given $\lambda,\rho>0$,
 $\vartheta_{\lambda,\rho}(x):=\lambda\|x\|_1-\lambda\rho^{-1}\sum_{i=1}^p\psi^*(\rho|x_i|)$
 for $x\in\mathbb{R}^p$. For a given penalty parameter $\sigma>0$,
 the augmented Lagrangian function of \eqref{smooth-surrogate} is defined as
 \[
   L_\sigma(x,z;y):=e_{\varepsilon}f(z)+\frac{\mu}{2}\|x\|^2+\vartheta_{\lambda,\rho}(x)
   +\langle y,Ax-b-z\rangle+ \frac{\sigma}{2}\|Ax-b-z\|^2.
 \]
 The iteration steps of our iPADMM for solving \eqref{smooth-surrogate} are
 described as follows.
%------------------------------------------------------------------------------------------
 \begin{algorithm}[H]
 \caption{\label{iPADMM}({\bf The iPADMM for solving \eqref{smooth-surrogate}})}
 \textbf{Initialization:} Select suitable $\lambda>0,\rho>0$ and $\varepsilon>0$.
                          Choose a penalty parameter $\sigma>0$ and a starting point
                          $(z^0,x^0,y^0)$. Set $\gamma=\frac{\sigma\|A\|^2}{2}+\frac{\lambda(a+1)\rho}{2(a-1)}-\mu$ and $k=0$.\\
 \textbf{while} the stopping conditions are not satisfied \textbf{do}
 \begin{enumerate}
 \item Compute the optimal solution of the following minimization problems
       \begin{subnumcases}{}
         \label{xk-ADMM}
         x^{k+1}=\mathop{\arg\min}_{x\in\mathbb{R}^p}L_{\sigma}(x,z^{k},y^k)+\frac{1}{2}\|x-x^k\|_{\gamma I-\sigma A^{\mathbb{T}}A}^2;\\
       \label{zk-ADMM}
         z^{k+1}=\mathop{\arg\min}_{z\in\mathbb{R}^p}L_{\sigma}(x^{k+1},z,y^k).
      \end{subnumcases}

      \vspace{-0.3cm}

  \item Update the Lagrange multiplier by $y^{k+1}=y^k+\sigma(Ax^{k+1}-z^{k+1}-b)$.

  \item Set~$k\leftarrow k+1$, and then go to Step 1.
  \end{enumerate}
 \end{algorithm}
 \begin{remark}
 {\bf(i)} Since $\gamma=\frac{\sigma\|A\|^2}{2}+\frac{\lambda(a+1)\rho}{2(a-1)}-\mu$,
  the matrix $\gamma I-\sigma A^{\mathbb{T}}A$ may be indefinite. So, Algorithm \ref{iPADMM}
  is the ADMM with an indefinite-proximal term. After a rearrangement,
  \begin{align*}
  x^{k+1}&=\mathop{\arg\min}_{x\in\mathbb{R}^p}\Big\{\frac{\mu+\gamma}{2}\|x-\xi^k\|^2+\vartheta_{\lambda,\rho}(x)\Big\}
  =\mathcal{P}_{(\mu+\gamma)^{-1}}\vartheta_{\lambda,\rho}(\xi^k),\\
  z^{k+1}&=\mathop{\arg\min}_{z\in\mathbb{R}^p}\Big\{\frac{\sigma}{2}\big\|z-\eta^k\big\|^2
  +e_{\varepsilon}f(z)\Big\}=\mathcal{P}_{\sigma^{-1}}(e_{\varepsilon}f)(\eta^k),
  \end{align*}
 where $\xi^k=\frac{1}{\mu+\gamma}\big[\gamma x^k+\sigma A^{\mathbb{T}}(z^k+b-Ax^k-\sigma^{-1}y^k)\big]$
 and $\eta^k=Ax^{k+1}\!-b+\frac{1}{\sigma}y^k$. Since $\theta_{\lambda,\rho}$ and
 $e_{\varepsilon}f(z)$ are separable, it is easy to achieve their proximal mappings
 and $(x^{k+1},z^{k+1})$.

 \medskip
 \noindent
 {\bf(ii)} By combining the optimality conditions of \eqref{xk-ADMM}-\eqref{zk-ADMM}
  and the multiplier update step, and comparing with the stationary point condition
  of \eqref{smooth-surrogate}, we terminate Algorithm \ref{iPADMM} at the iterate $(z^{k},x^{k},y^{k})$
  when $k>k_{\rm max}$ or $\max({\rm pinf}^{k},{\rm dinf}^{k})\le\epsilon_{\rm admm}$, where
  \[
   {\rm pinf}^{k}\!=\frac{\|y^{k}\!-\!y^{k-1}\|}{\sigma(1+\|b\|)},\
   {\rm dinf}^{k}\!=\frac{\|A^{\mathbb{T}}(y^{k}\!-\!y^{k-1})-\sigma A^{\mathbb{T}}(Ax^{k-1}\!-\!z^{k-1}\!-\!b)\!-\!\gamma(x^{k}\!-\!x^{k-1})\|}{1+\|b\|}.
  \]
 \end{remark}

  By the Lipschitz continuity of $\nabla e_{\varepsilon}f$ and the semi-convexity
  of $\theta_{\lambda,\rho}+\frac{\mu}{2}\|\cdot\|^2$, it is not difficult to obtain
  the following results for the sequence generated by Algorithm \ref{iPADMM}.
%------------------------------------------------------------------------------------------------
 \begin{lemma}\label{descent-iPADMM}
  Let $\{(x^k,z^k,y^k)\}_{k\in\mathbb{N}}$ be the sequence generated by
  Algorithm \ref{iPADMM}. Then,
 \begin{align*}
  & L_{\sigma}(x^k,z^k;y^k)-L_{\sigma}(x^{k+1},z^{k+1};y^{k+1})\\
  &\ge\Big[\frac{\sigma}{2}-\frac{4}{\sigma\varepsilon^2}\Big]\|z^{k+1}-z^{k}\|^2
      +\frac{\lambda(a+1)\rho-2(a\!-\!1)\mu}{4(a\!-\!1)}\|x^{k+1}-x^k\|^2\quad\ \forall k\in\mathbb{N},
 \end{align*}
 which implies that the sequence $\{L_{\sigma}(x^k,z^k,y^k)\}_{k\in\mathbb{N}}$
 is nonincreasing when $\sigma>2\sqrt{2}/\varepsilon$.
 Furthermore, when $\sigma>4/\varepsilon$, the sequence
 $\{L_{\sigma}(x^k,z^k,y^k)\}_{k\in\mathbb{N}}$ is lower bounded.
 \end{lemma}
%-------------------------------------------------------------------------------------
 \begin{lemma}\label{Subgrad-iPADMM}
  For each $k\in\mathbb{N}$, there exists a vector
  $\Delta u^{k+1}\in\partial L_{\sigma}(x^{k+1},z^{k+1};y^{k+1})$ with
  \[
   \|\Delta u^{k+1}\|\le\sqrt{3}\max(\|\gamma I\!-\!\sigma A^{\mathbb{T}}A\|,\sqrt{\|A\|^2\!+1},\sigma\|A\|)
   \|(x^{k+1}\!-x^k;z^{k+1}\!-z^k;y^{k+1}\!-y^k)\|.
 \]
 \end{lemma}

 Thus, by the semi-algebraic property of $L_{\sigma}$, one may follow
 the recipe in \cite{Attouch13} to obtain the global convergence
 of Algorithm \ref{iPADMM}. For the convergence analysis of the ADMM
 for nonconvex nonsmooth problems, the reader may refer to
 the related reference \cite{LiPong15,WangYZ19}.

 \medskip

 Next we focus on the solution of the subproblem \eqref{subprob-x}
 involved in Algorithm \ref{Alg}, which is the pivotal part for
 the implementation of Algorithm \ref{Alg}. Inspired by the numerical
 results reported in \cite{LiSun16}, we apply a dual semismooth Newton method
 for solving it.
%------------------------------------------------------------------------
 \subsection{Dual semismooth Newton method for \eqref{subprob-x}}\label{sec5.1}

 For each $k\ge 0$, write $h_k(x)\!:=\|\omega^k\circ x\|_1+\frac{1}{2}\mu\|x\|^2$
 with $\omega^k=\lambda(e-w^k)$. By introducing an additional variable $z\in\mathbb{R}^n$,
 the subproblem \eqref{subprob-x} can be equivalently written as
 \begin{align}\label{Esubprobj}
  &\min_{ x\in\mathbb{R}^p,z\in\mathbb{R}^n}\Big\{f (z)+h_k(x)+\frac{\gamma_{1,k}}{2}\| x- x^k\|^2+\frac{\gamma_{2,k}}{2}\|z-z^k\|^2\Big\}\nonumber\\
  &\quad\ {\rm s.t.}\quad A x-z-b=0\ \ {\rm with}\ \ z^k=\!A x^k\!-b.
 \end{align}
 After an elementary calculation, the dual of the problem \eqref{Esubprobj} takes the following form
 \begin{equation}\label{subdprobj}
   \min_{u\in\mathbb{R}^{n}}\bigg\{\Psi_{k}(u):=\frac{\|u\|^2}{2\gamma_{2,k}} -e_{\gamma_{2,k}^{-1}}f \Big(z^k+\frac{u}{\gamma_{2,k}}\Big)
   -e_{\gamma_{1,k}^{-1}}h_k\Big( x^k-\frac{A^{\mathbb{T}}u}{\gamma_{1,k}}\Big)
   +\frac{\|A^{\mathbb{T}}u\|^2}{2\gamma_{1,k}}\bigg\}.
 \end{equation}
 Clearly, the strong duality result holds for the problems \eqref{Esubprobj}
 and \eqref{subdprobj}. Since $\Psi_{k}$ is smooth and convex,
 seeking an optimal solution of \eqref{subdprobj} is equivalent to
 finding a root to
 \begin{equation}\label{semismooth-system}
   \Phi_{k}(u):=\nabla\Psi_k(u)
   =\mathcal{P}_{\gamma_{2,k}^{-1}}f \Big(z^k+\frac{u}{\gamma_{2,k}}\Big)
   -A\mathcal{P}_{\gamma_{1,k}^{-1}}h_k\Big( x^k\!-\!\frac{A^{\mathbb{T}}u}{\gamma_{1,k}}\Big)+b=0.
 \end{equation}
  Notice that $\Phi_{k}$ is strongly semismooth since
  $\mathcal{P}_{\gamma_{2,k}^{-1}}f $ and $\mathcal{P}_{\gamma_{1,k}^{-1}}h_k$
  are strongly semismooth by Section \ref{sec2.2} and the composition of
  strongly semismooth mappings is strongly semismooth by \cite[Proposition 7.4.4]{Facchinei03}.
  By \cite[Proposition 2.3.3 \& Theorem 2.6.6]{Clarke83}, we have
  \begin{align}\label{inclusion}
   \partial_C\Phi_{k}(u)&\subseteq
   \gamma_{2,k}^{-1}\partial_C\big[\mathcal{P}_{\gamma_{2,k}^{-1}}f \big]\Big(z^k\!+\!\frac{u}{\gamma_{2,k}}\Big)
   +\gamma_{1,k}^{-1}A\partial_C\big[\mathcal{P}_{\gamma_{1,k}^{-1}}h_k\big]
   \Big( x^k\!-\!\frac{A^{\mathbb{T}}u}{\gamma_{1,k}}\Big)A^{\mathbb{T}}\nonumber\\
   &=\gamma_{2,k}^{-1}\mathcal{U}_k(u)+\gamma_{1,k}^{-1}A\mathcal{V}_k(u)A^{\mathbb{T}}
   \quad\ \forall u\in\mathbb{R}^n
  \end{align}
  where, by Lemma \ref{CJacobi-h}, $\mathcal{U}_k(u)$ and $\mathcal{V}_k(u)$ take
  the following form with $\widetilde{\gamma}_{1,k}=\frac{\gamma_{1,k}}{\gamma_{1,k}+\mu}$:
  \begin{align*}
  \mathcal{U}_k(u)&:=\Big\{{\rm Diag}(v_1,\ldots,v_n)\ |\ v_i\in\partial_C\big[\mathcal{P}_{\gamma_{2,k}^{-1}}(n^{-1}\theta)\big](z_i^k+\gamma_{2,k}^{-1}u_i)\Big\},\\
  \mathcal{V}_k(u)&:=\Big\{{\rm Diag}(v_1,\ldots,v_n)\ |\ v_i=\widetilde{\gamma}_{1,k}\ {\rm if}\
   |(\gamma_{1,k} x^k\!-\!A^{\mathbb{T}}u)_i|>\omega_i^k,
   {\rm otherwise}\,v_i\in\big[0,\widetilde{\gamma}_{1,k}\big]\!\Big\}.
  \end{align*}
  For each $U\in\mathcal{U}_k(u)$ and $V\in\mathcal{V}_k(u)$,
  the matrix $\gamma_{2,k}^{-1}U+\gamma_{1,k}^{-1}AVA^{\mathbb{T}}$
  is positive semidefinite, and moreover, it is positive definite when
  $\big\{i\ |\ z_i^k+\gamma_{2,k}^{-1}u_i\in[-\frac{1}{n },\frac{1}{n}]\big\}=\emptyset$
  or the matrix $A_{J}$ has full row rank with $J=\{j\ |\ |(\gamma_{1,k} x^k\!-\!A^{\mathbb{T}}u)_j|>\omega_j^k\}$.
  Motivated by these facts, we apply the following semismooth Newton method to seeking
  a root of the system \eqref{semismooth-system}, which by \cite{QiSun93} is
  expected to have a superlinear even quadratic convergence rate.
%-------------------------------------------------------------------------------------
 \begin{algorithm}[H]
 \caption{\label{SNCG}{\bf\ \ A semismooth Newton-CG algorithm}}
 \textbf{Initialization:} Fix $k\ge 0$. Choose $\underline{\tau},\underline{\eta},\delta\!\in(0,1)$,
                          $c\in\!(0,\frac{1}{2})$ and $u^0=u^{k-1,*}$. Set $j=0$.\\
 \textbf{while} the stopping conditions are not satisfied \textbf{do}
 \begin{enumerate}
  \item  Choose $U^j\in\mathcal{U}_k(u^j)$ and $V^j\in\mathcal{V}_k(u^j)$
         and set $W^j=\gamma_{2,k}^{-1}U^j+\gamma_{1,k}^{-1}AV^jA^{\mathbb{T}}$.
         Solve
        \begin{equation*}\label{SNCG-dj}
         (W^{j}+\tau_jI)d=-\Phi_{k}(u^j)
         \end{equation*}
        with the conjugate gradient (CG) algorithm to find an approximate $d^j$ such that
        $\|(W^{j}+\tau_jI)d^j\|\le\min(\underline{\eta},\|\Phi_{k}(u^j)\|^{1.1})$,
        where $\tau_j=\min(\underline{\tau},\|\Phi_{k}(u^j)\|)$.

  \item Set $\alpha_j=\delta^{m_j}$, where $m_j$ is the first nonnegative integer $m$ satisfying
        \begin{equation*}
         \Psi_{k}(u^j+\delta^md^j)\leq\Psi_{k}(u^j)+c\,\delta^m\langle\nabla\Psi_{k}(u^j),d^j\rangle.
        \end{equation*}

  \item Set $u^{j+1}=u^j+\alpha_jd^j$ and $j\leftarrow j+1$, and then go to Step 1.
 \end{enumerate}
 \textbf{end while}
 \end{algorithm}	
 \begin{remark}\label{remark-SNCG}
  Fix an arbitrary $k\ge 0$. Let $u^{k,*}$ be a root to the semismooth system \eqref{semismooth-system}.
  Set $x^{k,*}\!=\mathcal{P}_{\gamma_{1,k}^{-1}}h_k( x^k\!-\!\frac{A^{\mathbb{T}}u^{k,*}}{\gamma_{1,k}})$
  and $z^{k,*}\!=\mathcal{P}_{\gamma_{2,k}^{-1}}f (z^k+\frac{u^{k,*}}{\gamma_{2,k}})$.
  Then $Ax^{k,*}\!-z^{k,*}\!-b=0$ and
  \begin{align*}
   & f(z^{k,*})+h_k(x^{k,*})+\frac{\gamma_{1,k}}{2}\|x^{k,*}- x^k\|^2
   +\frac{\gamma_{2,k}}{2}\|z^{k,*}-z^k\|^2+\Psi_{k}(u^{k,*})\\
   &=\langle z^k-z^{k,*},u^{k,*}\rangle+\langle x^k-x^{k,*},A^{\mathbb{T}}u^{k,*}\rangle.
  \end{align*}
  That is, $(x^{k,*},z^{k,*})$ is a feasible solution to \eqref{Esubprobj} and
  the gap between its objective value and the dual optimal value is
  $\langle z^k-z^{k,*},u^{k,*}\rangle+\langle x^k-x^{k,*},A^{\mathbb{T}}u^{k,*}\rangle$.
  This motivates us to terminate Algorithm \ref{SNCG} at $u^j$ when
  $j>j_{\rm max}$ or the following conditions are satisfied:
  \begin{equation}\label{stop-rule-SNCG}
    \frac{\|\Phi_{k}(u^j)\|}{1+\|b\|}\le\epsilon_{\rm SNCG}^k\ \ {\rm and}\ \
    \frac{|\langle z^k-z^{k,j},u^j\rangle+\langle  x^k-x^{k,j},A^{\mathbb{T}}u^j\rangle|}{1+\|b\|}
    \le \epsilon_{\rm SNCG}^k
  \end{equation}
  where $z^{k,j}\!=\mathcal{P}_{\gamma_{2,k}^{-1}}f(z^k+\frac{u^j}{\gamma_{2,k}})$
  and $x^{k,j}=\mathcal{P}_{\gamma_{1,k}^{-1}}h_k(x^k\!-\!\frac{A^{\mathbb{T}}u^j}{\gamma_{1,k}})$.
  \end{remark}
%------------------------------------------------------------------------
 \subsection{Numerical comparisons on synthetic and real data}\label{sec5.3}

  We compare the performance of Algorithm \ref{Alg} armed with
  Algorithm \ref{SNCG} solving the subproblems (PMMSN for short) with
  that of iPADMM (i.e., Algorithm \ref{iPADMM}) via synthetic and real data,
  in terms of the computing time, approximate sparsity and relative $\ell_2$-error.
  Among others, $N_{\rm nz}(x)\!:=\!\sum_{i=1}^p\mathbb{I}\big\{|x_i|>\!10^{-6}\|x\|_\infty\big\}$
  denotes the approximate sparsity of a vector $x$ and
  $\textbf{L2err}:=\frac{\|x^{\rm out}-x^*\|}{\|x^*\|}$ means
  the relative $\ell_2$-error of the output $x^{\rm out}$. All numerical tests
  are done with a desktop computer running on 64-bit Windows Operating System
  with an Intel(R) Core(TM) i7-7700 CPU 3.6GHz and 16 GB memory.

  \medskip

  Unless otherwise stated, the two solvers are using $a=6.0$ for $\phi$,
  $\mu=10^{-8}$, and the same $x^0$ yielded by applying Algorithm \ref{SNCG}
  to \eqref{x0sub} with $\epsilon_{\rm SNCG}^0\!=10^{-5}$ and $j_{\rm max}\!=50$.
  By Remark \ref{remark-Alg}(ii), we choose $\rho=\max(1,\frac{25}{6\|x^0\|_\infty})$
  for $n\le p$ and $\rho=\max(1,\frac{25}{4\|x^0\|_\infty})$ for the case $n>p$,
  i.e., a little larger $\rho$ for $n>p$ since now $\frac{1}{n}A^{\mathbb{T}}A$ is positive definite.
  The parameter $\lambda$ (or $\nu=\rho^{-1}\lambda$) and $\varepsilon$ are specified in the test examples.
  The other parameters of Algorithm \ref{Alg} are chosen as $\gamma_{1,0}=0.1,\gamma_{2,0}=0.1,
  \underline{\gamma_1}=\underline{\gamma_2}=10^{-8},\varrho=0.8$,
  and the parameter $\sigma$ of Algorithm \ref{iPADMM} is chosen as
  $\sigma={4.5}/{\varepsilon}$ by its convergence analysis. For Algorithm \ref{Alg},
  besides using the stopping criterion in Remark \ref{remark-Alg} with $k_{\rm max}=200$
  and ${\rm tol}=10^{-6}$, we also terminate it at $x^k$ when
  $|N_{\rm nz}(x^{k-j})-N_{\rm nz}(x^{k-j-1})|\le 2$ for $j=0,1,2$ and
  ${\bf Err}_k\le 10^{-4}$, and Algorithm \ref{SNCG} is using
  the stopping criterion in Remark \ref{remark-SNCG} with $j_{\rm max}=50$ and
  $\epsilon_{\rm SNCG}^k=\max(10^{-6},0.8\epsilon_{\rm SNCG}^{k-1})$
  for $\epsilon_{\rm SNCG}^{0}=10^{-5}$. For Algorithm \ref{iPADMM},
  we use the stopping criterion described in Section \ref{sec5.2} with
  $k_{\rm max}=20000$ and $\epsilon_{\rm admm}=10^{-5}$.

  \medskip

  Before testing, we take a closer look at the choice of $\varepsilon$ for
  Algorithm \ref{iPADMM}. Figure \ref{fig1} below shows that as $\varepsilon$ increases,
  the sparsity and relative $\ell_2$-error of the output
  $\overline{x}^{\varepsilon}$ of iPADMM for solving \eqref{smooth-surrogate}
  associated to $\varepsilon$ becomes better and keeps unchanged when $\varepsilon$
  is over a threshold, but the loss value $\frac{1}{n}\|A\overline{x}^{\varepsilon}-b\|_1$
  first increases and then decreases a certain level and keeps it unchanged.
  In view of this, we regard the smallest one among those $\varepsilon$ for which
  $\overline{x}^{\varepsilon}$ has the sparsity closest to that of the true
  $x^*$ for synthetic data (respectively, that of the output of PMMSN for real data)
  as the best, by noting that such $\overline{x}^{\varepsilon}$ usually has
  a favorable $\ell_2$-error and the model \eqref{smooth-surrogate}
  with a smaller $\varepsilon$ is closer to \eqref{Eprob-surrogate}.
  For the subsequent tests, we search such a best $\varepsilon_{\rm opt}$
  from an appropriate interval (specified in the examples) by comparing the sparsity
  of $\overline{x}^{\varepsilon}$ corresponding to $20$ $\varepsilon$'s.
  To search the best $\varepsilon_{\rm opt}$ in this way, despite of impracticality,
  is just for numerical comparison.
  \vspace{-0.3cm}
%------------------------------------------------------------------------------------
 \begin{figure}[H]
  \centering
  \includegraphics[height=6cm,width=6.2in]{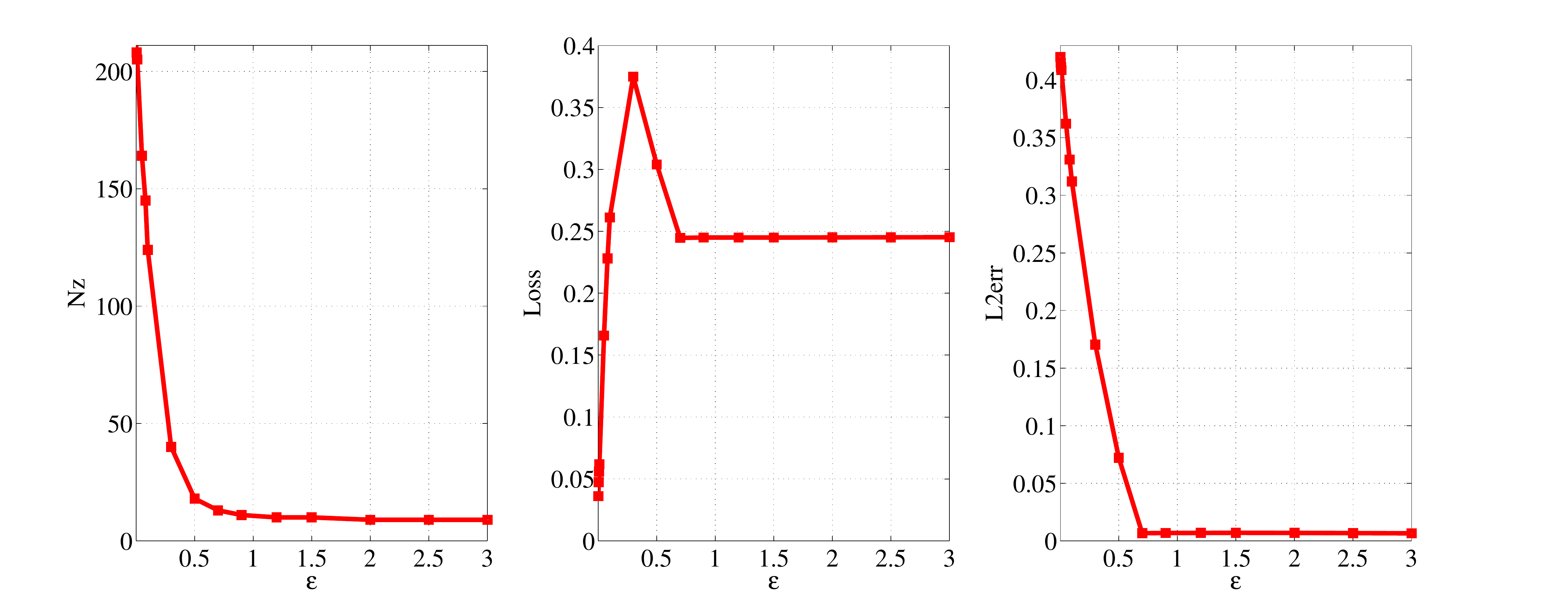}
  \caption{The influence of $\varepsilon$ on sparsity and loss value
   for Example \ref{exam2} with $|\mathcal{I}|=\lfloor 0.5n\rfloor$}
  \label{fig1}
   \vspace{-0.3cm}
 \end{figure}

%----------------------------------------------------------------------------------
 \subsubsection{Synthetic data examples}\label{subsec5.3.1}

 We use some synthetic data $(b,A)$ to evaluate the performance of PMMSN and iPADMM
 for solving the surrogate problem \eqref{Eprob-surrogate}. The data $b$ is given
 by \eqref{observe} with each $a_i^{\mathbb{T}}\sim N(0,\Sigma)$
 and a nonzero noise vector $\varpi$ whose nonzero entries are i.i.d.,
 where the covariance matrix $\Sigma\in\mathbb{R}^{p\times p}$ and the distribution
 of the nonzero entries of $\varpi$ are specified in the examples.
%----------------------------------------------------------------------------------

 \medskip
 \noindent
 {\bf 1. Influence of the sparsity of $\varpi$ on the solvers}

 \medskip

 The following example involves the noise $\varpi$ with $\varpi_{\mathcal{I}}$
 following the normal distribution. We use it to test the performance of
 two solvers under different sparsity of $\varpi$.
%------------------------------------------------------------------------------
 \begin{example}\label{exam1}
  (see \cite{Gu18}) Take $(n,p)=(200,1000), \Sigma_{i,j}=0.8^{|i-j|}$
  and $\varpi_{\mathcal{I}}\sim N(0,2I)$. The true $x^*$ has the form of
  $(2,\,0,\,1.5,\,0,\,0.8,\,0,\,0,\,1,\,0,\,1.75,\,0,\,0,\,0.75,\,0,\,0,\,0.3,\,{\bf 0}_{p-16}^{\mathbb{T}})^{\mathbb{T}}$.
 \end{example}

 Figure \ref{fig2} below plots the relative $\ell_2$-error and time curves
 of two solvers under different sparsity rate (i.e., $|\mathcal{I}|/n$) of
 the noise vector $\varpi$, where PMMSN is solving the surrogate problem
 \eqref{Eprob-surrogate} with $\lambda=\max(0.05,0.2n^{-1}\!\interleave\!A\interleave_1)$
 and iPADMM is solving its smoothed form \eqref{smooth-surrogate} with
 the same $\lambda$ and $\varepsilon_{\rm opt}=0.7$.
 We see that as the sparsity rate increases, the relative $\ell_2$-error of two solvers
 increases, but the $\ell_2$-error of PMMSN is always lower than that of iPADMM
 and their difference is remarkable after $|\mathcal{I}|/n>0.3$. This is consistent
 with the conclusion of Theorem \ref{error-bound1} by Remark \ref{remark-errbound}(ii).
 In addition, as $|\mathcal{I}|/n$ increases, the time required by PMMSN has a small
 change but the time of iPADMM has a remarkable increase.
 This shows that PMMSN has a better performance for this example.
 %-----------------------------------------------------------------------------figure
 \begin{figure}[H]
  \centering
 \includegraphics[height=7cm,width=6.2in]{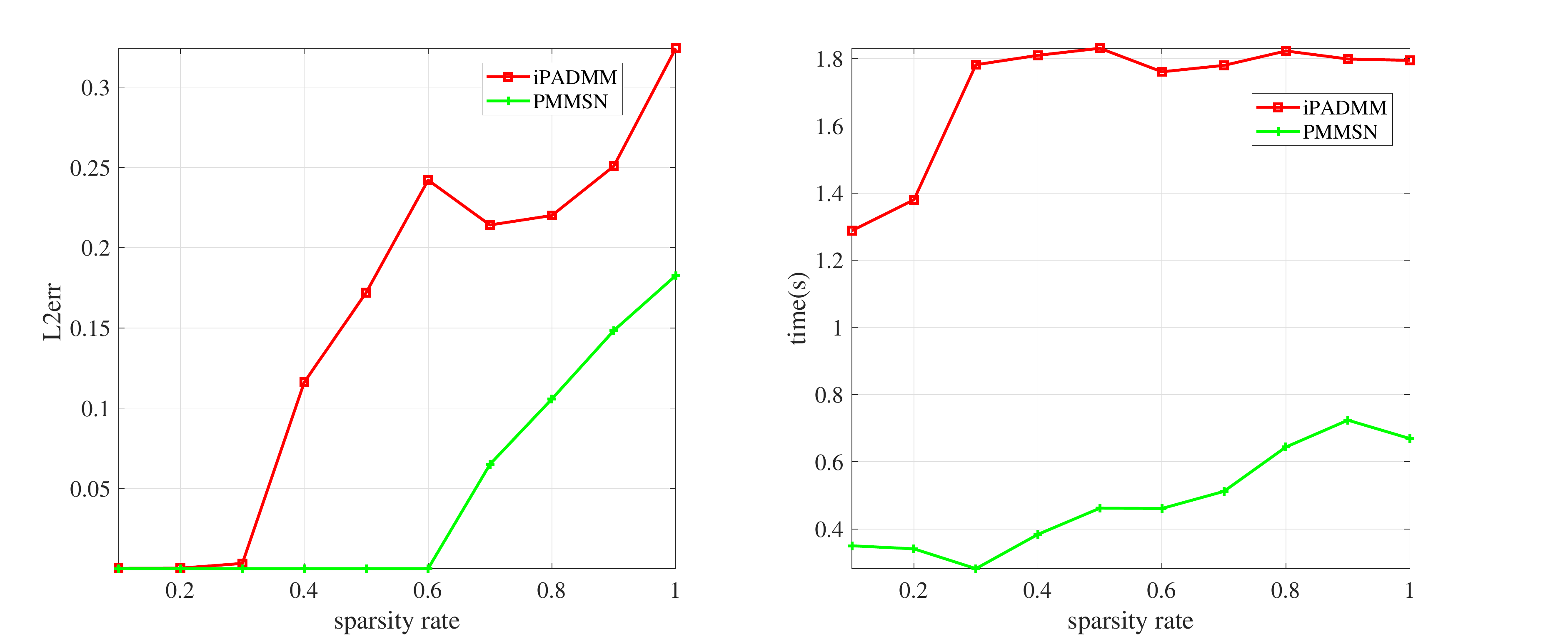}
  \caption{The relative $\ell_2$-error and computing time of two solvers
  under different $|\mathcal{I}|$}
  \label{fig2}
  \vspace{-0.2cm}
 \end{figure}

%----------------------------------------------------------------------------------
 \medskip
 \noindent
 {\bf 2. Influence of the parameter $\lambda$ on the solvers}

  \medskip

 The following example is also from \cite{Gu18} which involves
 a heavy-tailed noise. We employ it to test the performance of
 two solvers under different $\lambda$ or $\nu=\lambda/\rho$.
%------------------------------------------------------------------------------
 \begin{example}\label{exam2}
  Be same as Example \ref{exam1} except that
  $\varpi_{\mathcal{I}}=\!\frac{\|Ax^*\|}{3\|\xi_{\mathcal{I}}\|}\xi_{\mathcal{I}}$
  with $|\mathcal{I}|=\lfloor0.5n\rfloor$ and all entries of $\xi_{\mathcal{I}}$
  follow the Cauchy distribution of density $d(u)=\frac{1}{\pi(1+u^2)}$.
 \end{example}

 \vspace{-0.2cm}
%-----------------------------------------------------------------------------figure
  \begin{figure}[H]
  \centering
 \includegraphics[height=7cm,width=6.0in]{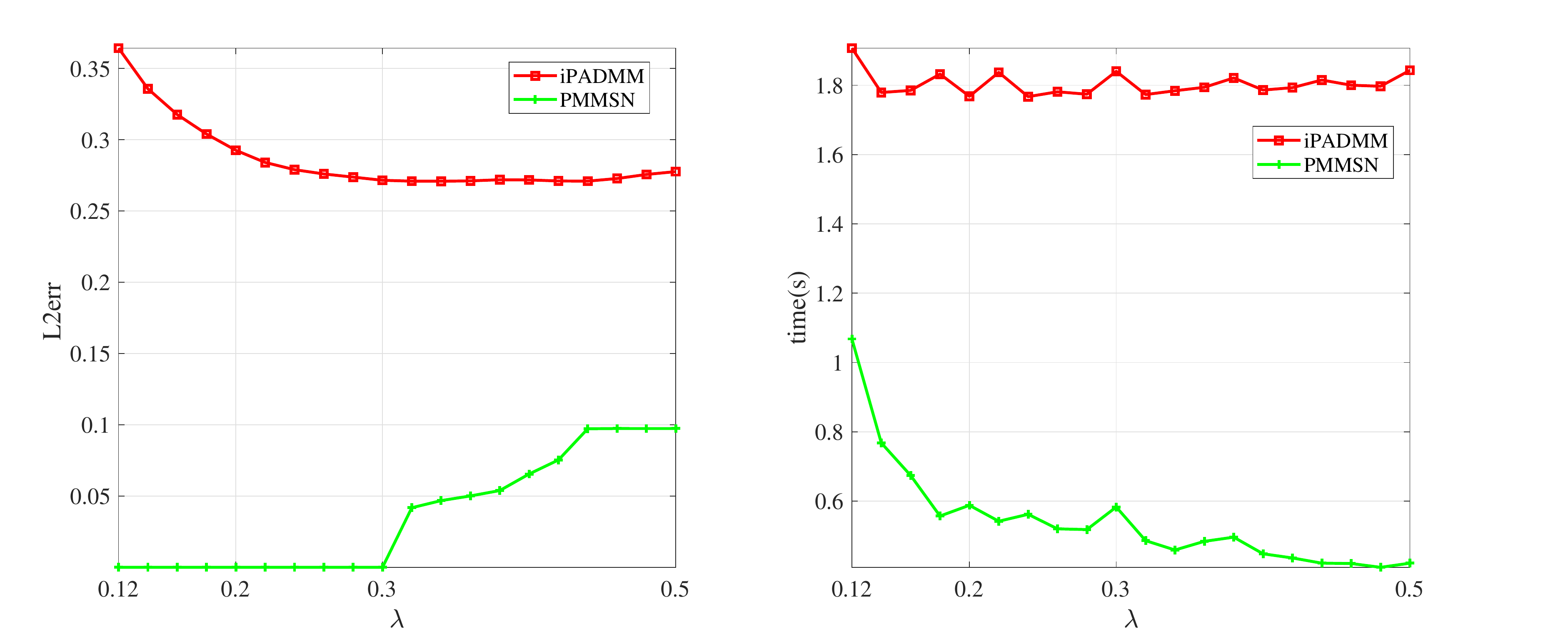}
  \caption{The relative $\ell_2$-error and computing time of two solvers under different $\lambda$}
  \label{fig3}
 \end{figure}

  \medskip

 Figure \ref{fig3} plots the relative $\ell_2$-error and computing time curves
 of two solvers under different $\lambda$, where iPADMM is solving \eqref{smooth-surrogate}
 with $\varepsilon_{\rm opt}=0.7$. We see that as the parameter $\lambda$ increases,
 the relative $\ell_2$-error of iPADMM first decreases and then increases slightly,
 whereas the relative $\ell_2$-error of PMMSN has a remarkable increase after
 $\lambda>0.3$ but is still much lower than that of iPADMM. This means that
 if the sparsity of the noise $\varpi$ is well controlled, the increase of $\lambda$
 in a certain range has a small influence on the error bound of the output of PMMSN.
 In addition, as $\lambda$ increases, the time of iPADMM has a tiny change,
 but the time of PMMSN becomes less and is always less than that of iPADMM.
 This shows that PMMSN has a better performance for this class of noise.

%----------------------------------------------------------------------------------
 \medskip
 \noindent
 {\bf 3. Performance of two solvers for other sparse noises}

 \medskip

 We test the performance of two solvers for other types of sparse noises via
 six examples, generated randomly with $p=5000,s^*=\lfloor \sqrt{p}/2\rfloor$
 and $n=\lfloor 2s^*\ln(p)\rfloor$. The sparsity of the noise vector $\varpi$
 is always set as $|\mathcal{I}|=\lfloor 0.3n\rfloor$ and the nonzero entries
 of $x^*$ follow $N(0,4)$. The covariance matrix $\Sigma$ includes
 the autoregressive structure $\Sigma=(0.5^{|i-j|})_{ij}$ denoted by
 ${\rm AR}_{0.5}$ and the compound symmetric structure
 $\Sigma=(\alpha+(1\!-\!\alpha)\mathbb{I}_{\{i=j\}})$ for $\alpha=0.6$,
 denoted by ${\rm CS}_{0.6}$. Such covariance matrices are highly relevant
 although they are positive definite. The nonzero entries of $\varpi$ obey
 the following distributions: {\bf (1)} the normal distribution $N(0,100)$;
 {\bf (2)} the scaled Student's $t$-distribution with $4$ degrees of freedom
 $\sqrt{2}\times t_4$; {\bf (3)} the mixture normal distribution
 $N(0,\sigma^2)$ with $\sigma\sim {\rm Unif}(1,5)$, denoted by ${\rm MN}$;
 {\bf (4)} the Laplace distribution with density $d(u)=0.5\exp(-|u|)$;
 and {\bf (5)} the Cauchy distribution with density $d(u)=\frac{1}{\pi(1+u^2)}$.
 Table \ref{result-syn2} reports the average of the loss values,
 relative $\ell_2$-errors and computing time of total $\textbf{10}$ test problems
 for each example with a fixed $\lambda$, for which we always take
 $\lambda=\max(0.05,0.12n^{-1}\!\interleave\!A\interleave_1)$ by
 the assumption on $\lambda$ in Theorem \ref{error-bound1}.
 In Table \ref{result-syn2}, a$=$iPADMM and b$=$PMMSN,
 $\textbf{Nz}=N_{\rm nz}(x^{\rm out})$ denotes the approximate zero-norm of $x^{\rm out}$,
 \textbf{Loss} means the loss value $\frac{1}{n}\|Ax^{\rm out}\!-b\|_1$, $\|A\|^2$ means
 the average of the square spectral norm of $A$ for $10$ test problems, ${\bf FP}$
 and ${\bf FN}$ respectively represents the number of false positives and false zeros
 of $x^{\rm out}$, and $\varepsilon$ column lists the interval to search the best
 $\varepsilon_{\rm opt}$. During the testing, we find that when $\Sigma$ takes ${\rm CS}_{0.6}$,
 the noise of Cauchy distribution will have a norm over $10^3$, for which the two solvers
 both fail to yielding a reasonable solution. In view of this, we choose the $\textbf{10}$
 test problems generated randomly with $\|\varpi\|_{\infty}<1000$ for testing.
%-----------------------------------------------------------------------------------
 \begin{table}[H]
  \caption{The performance of iPADMM and PMMSN for sparse noise}
  \label{result-syn2}
  \centering
 \scalebox{0.65}{
 \begin{tabular}{l|c|c|c|c|c|c|c|c|c|c}
 \hline
 \hline
  \multirow{2}*{\qquad Problem} &\multirow{2}*{$\|A\|^2$} &\multirow{2}*{$\varepsilon$}
  & \multirow{2}*{$ \varepsilon_{\rm opt}$}& \multirow{2}*{$\|\varpi\|_{\infty}$}& {\bf Nz} & {\bf Loss}
  & {\bf L2err} & {\bf FP}& {\bf FN}&{\bf Time(s)} \\
  \qquad\ \ $\Sigma$|$\varpi$&    &     &      &  &  $a|b$  & $a|b$   & $a|b$   &$a|b$  &$a|b$ &$a|b$    \\
 \hline
 ${\rm AR}_{0.5}$|$N(0,100)$  & 1.08e+4 & [15,30] & 25 &9.90 & 38.1|35.0 &1.501|1.489 & 2.34e-3|5.68e-7&3.2|0&0.1|0&42.4|45.9\\
 \hline
 ${\rm AR}_{0.5}$|$\sqrt{2}\times t_4$ &1.08e+4&[10,30]&15&12.4 & 39.6|35.0& 0.464|0.448& 3.01e-3|2.21e-6&4.7|0&0.1|0&44.5|42.3\\
 \hline
  ${\rm AR}_{0.5}$|MN &1.08e+4&[10,30] &   20  &12.07&   36.2|35.0 &  0.734|0.725   &   1.73e-3 |1.68e-6  & 1.3|0 &   0.1|0  & 46.7|53.5\\
 \hline
  ${\rm AR}_{0.5}$|Laplace &1.08e+4&[10,30]& 15 &6.32 &  35.7|35.0&  0.301|0.294&  1.18e-3|8.21e-6 &  0.8|0&  0.1|0 & 43.8|49.8 \\
 \hline
 ${\rm AR}_{0.5}$|Cauchy &1.08e+4&[20,35]& 27 &322.7 & 90.5|34.0& 2.208|1.477&2.81e-1|9.96e-3 & 60.3|0&4.8|1 &42.1|108.5 \\
 \hline
 ${\rm CS}_{0.6}$|$N(0,100)$ &1.77e+6&[1600,2000] &1800& 9.90 & 34.5|34.8&2.907|1.504& 3.59e-1|4.54e-3&11.6|0.1&12.1|0.3& 50.6|28.6\\
 \hline
 ${\rm CS}_{0.6}$|$\sqrt{2}\times t_4$&1.77e+6&[1000,1500]& 1225 &12.4 & 39.3|34.9&1.571|0.462&2.44e-1|4.07e-3 &12.1|0&7.8|0.1&50.4|19.5\\
 \hline
  ${\rm CS}_{0.6}$|MN &1.77e+6&[1000,1500]& 1350  &12.07 &  39.5|34.6&  1.968|0.753& 2.84e-1|8.19e-3  &  13.2|0& 8.7|0.4 & 59.5|24.9 \\
 \hline
  ${\rm CS}_{0.6}$|Laplace &1.77e+6&[1000,1500]& 1150 &6.32 &  36.3|35.1&  1.341|0.302& 2.15e-1|2.24e-3  &  8.6|0.2& 7.3|0.1 & 60.2|22.5 \\
\hline
 ${\rm CS}_{0.6}$|Cauchy & 1.80e+6 & [1200,1800] & 1500&322.7 & 226.7|29.4&3.523|1.828&9.35e-1|9.28e-2 &209.3|0.2&17.6|5.8&51.1|93.8 \\
\hline
\hline
\end{tabular}}
\end{table}

 From Table \ref{result-syn2}, we see that the average sparsity yielded by PMMSN and iPADMM for
 all test examples except for ${\rm AR}_{0.5}$|Cauchy and ${\rm CS}_{0.6}$|Cauchy is
 close to that of the true vector $x^*$, but the average relative $\ell_2$-error and
 $({\bf FP},{\bf FN})$ yielded by iPADMM are worse than those yielded by PMMSN, especially for
 those examples of ${\rm CS}_{0.6}$. For the most difficult ${\rm CS}_{0.6}$|Cauchy,
 the average sparsity, relative $\ell_2$-error and $({\bf FP},{\bf FN})$ given by
 iPADMM are much worse than those yielded by PMMSN since, the parameter
 $\varepsilon$ is very sensitive to the data and the best $\varepsilon_{\rm opt}$
 is not suitable for all $10$ test problems. This shows that replacing $f$ with
 a smooth approximation $e_{\varepsilon}f$ is not effective for highly relevant
 covariance matrix $\Sigma$ and heavy-tailed sparse noises, though $\varepsilon$
 is elaborately selected.

%---------------------------------------------------------------------------
 \subsubsection{Real data examples}\label{subsec5.3.2}

 This part uses the LIBSVM datasets from \url{https://www.csie.ntu.edu.tw}
 to test the efficiency of PMMSN for large scale problems. For those data sets
 with a few features, such as \textbf{pyrim, abalone, bodyfat, housing, mpg, space ga},
 we follow the same line as in \cite{Tang19} to expand their original features
 by using polynomial basis functions over those features. For example,
 the last digit in \textbf{pyrim5} indicates that a polynomial of order $5$
 is used to generate the basis function. Such a naming convention is also
 applicable to the other expanded data sets. These data sets are quite difficult
 in terms of the dimension and the largest eigenvalues of $A^\mathbb{T}A$.
 Table \ref{result-lib} reports the numerical results of two solvers to
 solve the corresponding problems
 with $\lambda=\max(0.05,0.1n^{-1}\!\interleave\!A\interleave_1)$.
%---------------------------------------------------------------------------
 \begin{table}[H]
  \caption{The performance of two solvers for LIBSVM datasets}
  \label{result-lib}
  \centering
 \scalebox{0.7}{
 \begin{tabular}{c|c|c|c|c|c|c|c|c}
 \hline
\hline
\multirow{2}*{Name of data}  & \multirow{2}*{$(n,p)$} & \multirow{2}*{$\|A\|^2$ }     & \multirow{2}*{$\lambda$}& \multirow{2}*{$ \varepsilon$}& \multirow{2}*{$\varepsilon_{\rm opt}$} & NZ        & Loss             & Time(s) \\
                             &             & &           &                                                     &                           &  $a|b$    & $a|b$            &$a|b$     \\
%\hline
%E2006.train         & 16087;150360  & $1.91e+05$       &  0.3392  & [0.1,10]&   5    &  1|1     & 0.2661|0.2661  &   10|11      \\
\hline
E2006.test          & 3308;150358   & 4.79e+4       &  0.3776  & [0.1,10]&  2.5   & 1|1       & 0.2361|0.2361 &   42|42       \\
%\hline
%log1p.E2006.train   & 16087;4272227 & $5.86e+07$       & 0.6080        & -&    1000           & 8|1   &  0.2729|0.2661  &   9584|18943     \\
\hline
log1p.E2006.test    & 3308;4272226  & 1.46e+7       & 0.5395          & [1500,2500]&       2000        &  1|1   & 0.2362|0.2361   &   7071|3208      \\
\hline
abalone7            & 4177;6435     &  5.23e+5      &  0.1      &[500,1500] &   1250          & 28|20   & 1.5140|1.4800  & 433|165       \\
\hline
bodyfat7            & 252;116280    &  5.30e+4      &  0.1      & [0.1,1]   &   0.3           &  3|3    & 7.45e-4|4.50e-4 &   604|466   \\
\hline
housing7            & 506;77520     &  3.28e+5      &  0.1      & [200,600]&   470      & 137 |139   & 2.6596|1.0204  &  791|604    \\
\hline
mpg7                & 392;3432      &  1.30e+4      &  0.1      &[50,150] &  90    & $ 15|15 $   & $1.9826|1.7744$  &   $ 20|4$     \\
\hline
pyrim5              & 74;201376     &  1.22e+6      &  0.1      &[15,150] &  60   & $ 21|21 $   & $0.0783|0.0185$  &   $ 349|210$     \\
\hline
space ga9           & 3107;5005     &  4.10e+3      &  0.1      &[0.1,10] &  5    & $ 5|5 $   & $ 0.1040|0.0982$  &   $ 261|127$     \\
\hline
\hline
\end{tabular}}
\end{table}

 \medskip

 From Table \ref{result-lib}, PMMSN works well in solving large scale
 difficult problems. Although the sparsity of its output is very close to
 that of the output of iPADMM, the loss value of its output is lower than
 that of the output of iPADMM. From the numerical tests for synthetic
 example, the loss value is usually consistent with the relative error.
 This means that the output yielded by PMMSN has better quality.
 In particular, the computing time required by PMMSN is less than
 the time required by iPADMM.

 \medskip

 From the above numerical comparisons, we conclude that PMMSN has
 an advantage in the quality of solutions and computing time, and
 it is robust under the scenario where $a_i^{\mathbb{T}}$ has
 a highly-relevant covariance and the noise is heavily-tailed,
 while the performance of iPADMM depends much on the smoothing
 parameter $\varepsilon$, and for those tough examples,
 the parameter $\varepsilon$ is very sensitive to the data.
 In addition, from the numerical results on synthetic examples,
 we find that when the sparsity of the noise $\varpi$ attains
 a certain level, say, $|\mathcal{I}|\le 0.6n$ for Example \ref{exam1}
 and $|\mathcal{I}|\le 0.3n$ for Example \ref{exam2} and the first
 fourth examples in Table \ref{result-syn2}, the relative
 $\ell_2$-error has an order about $10^{-6}$ which is close to
 the exact recovery. Then, it is natural to ask for which kind of
 covariance matrix $\Sigma$ and noises, the exact recovery of the limit
 $\overline{x}$ can be achieved by controlling the sparsity of the noise
 vector $\varpi$. We leave this interesting question for a future research topic.

%--------------------------------------------------------------------------
 \section{Conclusions}\label{sec6}

  For the zero-norm regularized PLQ composite problem, we have shown that
  its equivalent MPEC is partially calm over the set of global optima and
  obtain a family of equivalent DC surrogates by using this important property,
  which greatly improves the result of \cite[Theorem 3.2]{LiuBiPan18} for
  this class of problems. We develop a proximal MM method for solving
  one of the DC surrogates, and provide its theoretical certificates by
  establishing its global and local linear convergence, analyzing when
  the limit of the generated sequence is a local minimum, and deriving
  the error bound of the limit to the true vector for the data $(b,A)$
  from a linear observation model. Numerical comparisons with the convergent
  iPADMM for synthetic and real data examples verify that the proximal MM method
  armed with a dual semismooth Newton method for solving the subproblems
  has an advantage in the quality of solutions and the computing time,
  and is robust for the case where $A$ has a large spectral norm
  and $b$ is corrupted by the heavily-tailed noise; whereas the convergent
  iPADMM for the partially smoothed surrogate is ineffective for those
  tough test examples even with an elaborate selection of the smoothing
  parameter $\varepsilon$.

  \bigskip
 \noindent
 {\large\bf Acknowledgements}\ \ The authors would like to express their sincere thanks
 to Prof. Kim-Chuan Toh from National University of Singapore for helpful suggestions
 on the implementation of Algorithm \ref{SNCG} when visiting SCUT in March of 2019,
 and give thanks to Prof. Liping Zhu from RenMin University of China for helpful discussion
 on the result of Theorem \ref{error-bound1}. The research of Shaohua Pan and
 Shujun Bi is supported by the National Natural Science Foundation of China under project
 No.11971177 and No.11701186.

 {}

  \bigskip
  \noindent
  {\bf\large Appendix A}
 %------------------------------------------------------------------------------
 \begin{alemma}\label{alemma1}
  Fix any $x\in\mathbb{R}^p$ and $\rho>0$.
  Let $w_{\rho}\!:\mathbb{R}^p\to\mathbb{R}^p$ be the mapping
  in \eqref{wrho}. Then,
  \[
    \nabla g_{\rho}(x)=w_{\rho}(x)\circ{\rm sign}(x)\quad\ \forall x\in\mathbb{R}^p.
  \]
  %Moreover, for any $d\in\mathbb{R}^n$,
%  \[
%    g_{\rho}''(x;d):=\lim_{\tau\downarrow 0}\frac{g_{\rho}(x+\tau d)-g_{\rho}(x)-\tau\langle\nabla g_{\rho}(x),d\rangle}{\frac{1}{2}\tau^2}\ge 0.
%  \]
 \end{alemma}
 \begin{proof}
  By the expressions of $g_{\rho}$ and $w_{\rho}$ in equation \eqref{grho}-\eqref{wrho},
  it suffices to argue that
  $\rho^{-1}\varphi_{\rho}'(x_i)\!=\!(\psi^*)'(\rho|x_i|){\rm sign}(x_i)$
  for each $i$. By the expression of $\psi^*$, for any $t\in\mathbb{R}$,
  \begin{equation}\label{der-varphi-rho}
    \varphi_{\rho}'(t)=\left\{\begin{array}{cl}
                      0 & \textrm{if}\ |t|\leq \frac{2}{\rho(a+1)};\\
                      \frac{\rho((a+1)\rho|t|-2){\rm sign}(t)}{2(a-1)} & \textrm{if}\ \frac{2}{\rho(a+1)}<|t|\leq \frac{2a}{\rho(a+1)};\\
                      \rho{\rm sign}(t) & \textrm{if}\ |t|>\frac{2a}{\rho(a+1)}.
                \end{array}\right.
 \end{equation}
  On the other hand, by the expression of $\psi^*$ in \eqref{psi-star},
  it is easy to check that
  \vspace{-0.2cm}
  \begin{equation*}%\label{temp-aequa}
   (\psi^*)'(\rho|x_i|){\rm sign}(x_i)
   =\left\{\begin{array}{cl}
     0 & \textrm{if}\ |x_i|\leq \frac{2}{\rho(a+1)},\\
     \frac{(\rho(a+1)|x_i|-2){\rm sign}(x_i)}{2(a-1)} & \textrm{if}\ \frac{2}{\rho(a+1)}<|x_i|\leq \frac{2a}{\rho(a+1)},\\
     {\rm sign}(x_i) & \textrm{if}\ |x_i|>\frac{2a}{\rho(a+1)}.
   \end{array}\right.
  \end{equation*}
 By comparing $\rho^{-1}\varphi_{\rho}'(x_i)$ with $(\psi^*)'(\rho|x_i|){\rm sign}(x_i)$,
 the stated equality holds.
 \end{proof}

 \medskip
 \noindent
 {\bf\large Proof of Proposition \ref{prop-Theta}}
 \begin{aproof}
  {\bf(i)} The lower boundedness and coerciveness of $\Theta_{\lambda,\rho}$
  follows by the expressions of $\psi^*$ and the lower boundedness on $f$.
  By equation \eqref{der-varphi-rho}, a simple calculation shows $\varphi_{\rho}'$ is
  Lipschitz continuous in $\mathbb{R}$ of modulus $\rho^2\max(\frac{a+1}{2},\frac{a+1}{2(a-1)})$.
  Notice that $g_{\rho}(x)=\rho^{-1}\sum_{i=1}^p\varphi_{\rho}(x_i)$. Then,
  $-\nabla g_{\rho}$ is globally Lipschitz on $\mathbb{R}^p$ with the same Lip-constant.
  By invoking the descent lemma, this means that $-g_{\rho}$ is semiconvex
  of modulus $\rho^2\max(\frac{a+1}{2},\frac{a+1}{2(a-1)})$
  and so is the function $\Theta_{\lambda,\rho}$.

  \medskip
  \noindent
  {\bf(ii)} Since $\Theta_{\lambda,\rho}$ is semiconvex,
  the first two equalities follow by Remark \ref{remark-Gsubdiff}(iii).
  By using the smoothness of $g_{\rho}$ and \cite[Exercise 8.9]{RW98},
  it follows that
  \[
    \partial\Theta_{\lambda,\rho}(x)
    =\partial(F_{\mu}(\cdot)+\lambda\|\cdot\|_1+\delta_{\Omega}(\cdot))(x)-\lambda\nabla g_{\rho}(x).
  \]
  In addition, since ${\rm dom}f=\mathbb{R}^p$,
  from \cite[Theorem 23.8]{Roc70} it follows that
  \[
    \partial(F_{\mu}(\cdot)+\lambda\|\cdot\|_1+\delta_{\Omega}(\cdot))(x)
    =A^{\mathbb{T}}\partial\!f(Ax\!-b)+\mu x+\lambda\partial\|x\|_1 +\mathcal{N}_{\Omega}(x).
  \]
  The result directly follows from the last two equations.

  \medskip
  \noindent
  {\bf(iii)} Since $x\in\mathbb{R}^p$ is a critical point of
  $\Theta_{\lambda,\rho}$, from part (ii) and Lemma \ref{alemma1} it follows that
  \begin{align*}
    0&\in A^{\mathbb{T}}\partial\!f(Ax\!-b)+\mu x +\mathcal{N}_{\Omega}(x)
      +\lambda\partial\|x\|_1-\lambda\nabla g_{\rho}(x)\\
    &=A^{\mathbb{T}}\partial\!f(Ax\!-b)+\mu x+\mathcal{N}_{\Omega}(x)+\lambda\big[(1\!-\![w_\rho(x)]_1)\partial|x_1|
     \times\cdots\times(1\!-\![w_\rho(x)]_p)\partial|x_p|\big]
  \end{align*}
  where $w_{\rho}\!:\mathbb{R}^p\to\mathbb{R}^p$ is the mapping defined
  in equation \eqref{wrho}, and the equality is using the fact that $[w_{\rho}(x)]_i=0$
  if $x_i=0$ and $\partial|x_i|=\{{\rm sign}(x_i)\}$ if $x_i\ne 0$. Notice that
  $[w_{\rho}(x)]_i=(\psi^*)'(\rho|x_i|)=\min\big[1,\max\big(0,\frac{(a+1)\rho|x_i|-2}{2(a-1)}\big)\big]$
  for each $i$. Since $|x_{\rm nz}|\ge \frac{2a}{\rho(a+1)}$ by the given assumption,
  we have $[w_{\rho}(x)]_i=1$ for all $i\in{\rm supp}(x)$. Combining this with the last equation
  and the subdifferential characterization of $\|\cdot\|_0$ in \cite{Le13}, we conclude that
  \[
    0\in A^{\mathbb{T}}\partial\!f(Ax\!-b)+\mu x+\mathcal{N}_{\Omega}(x) +\partial\|x\|_0.
  \]
  The right hand side of the last inclusion is exactly the regular subdifferential
  of the objective function of \eqref{prob}.
  Thus, $x$ is a regular critical point of \eqref{prob}.

  \medskip
  \noindent
  {\bf(iv)} Since the set $\Omega$ is polyhedral and the function $F_{\mu}+g_{\rho}$
  is continuous and piecewise linear-quadratic, it follows that
  $\Theta_{\lambda,\rho}=F_{\mu}+g_{\rho}+\delta_{\Omega}$ is a piecewise
  linear-quadratic function with ${\rm dom}\Theta_{\lambda,\rho}=\Omega$.
  By \cite[Definition 10.20]{RW98}, there exist $p\times p$ symmetric matrices
  $M^1,\ldots,M^m$, vectors $b^1,\ldots,b^m\in\mathbb{R}^p$ and
  scalars $c_1,\ldots,c_m$ such that
  \[
    \Theta_{\lambda,\rho}(z)=\min_{1\le i\le m}\!\Big\{\langle z,M^iz\rangle+\langle b^i,z\rangle+c_i +\delta_{\mathcal{P}_i}(z)\Big\}\quad\ \forall z\in\mathbb{R}^p
  \]
  where each $\mathcal{P}_i$ is a polyhedral set. Notice that $\delta_{\Omega}$
  is continuous on the set $\Omega$ by the convexity of $\Omega$ and
  ${\rm ri}({\rm dom}\,\delta_{\Omega})=\Omega$. This means that
  $\Theta_{\lambda,\rho}=F_{\mu}+g_{\rho}+\delta_{\Omega}$ is continuous on
  the set $\Omega$. In addition, by Proposition \ref{prop-Theta}(ii),
  we know that ${\rm dom}\partial\Theta_{\lambda,\rho}\subseteq\Omega$.
  The two sides show that $\Theta_{\lambda,\rho}$ is continuous on
  ${\rm dom}\partial\Theta_{\lambda,\rho}$.
  Now by invoking \cite[Corollary 5.2]{LiPong18} and the last equation,
  we obtain the desired result. The proof is then completed.
%  Since $F_{\mu}+\delta_{\Omega}$ is piecewise linear-quadratic,
%  its graph is a semialgebraic set in $\mathbb{R}^{p+1}$,
%  which implies that it is semialgebraic. Since $t\mapsto \rho^{-1}\psi^*(\rho|t|)$
%  is semialgebraic, $-g_{\rho}$ is semialgebraic.
%  Clearly, the $\ell_1$-norm is semialgebraic. Thus,
%  $\Theta_{\lambda,\rho}$ is semialgebraic.
 \end{aproof}

 \bigskip
 \noindent
 {\bf\large Appendix B}
%------------------------------------------------------------------------
 \begin{alemma}\label{theta-lemma}
  Let $\theta\!:\mathbb{R}^n\to\mathbb{R}_{+}$ be a convex function
  with $\theta(0)=0$. For any given $\overline{t}\in\mathbb{R}$,
  \[
    \partial(\theta^2)(\overline{t})
    =\left\{\begin{array}{cl}
             \{0\} & {\rm if}\ \theta(\overline{t})=0;\\
          \theta(\overline{t})\partial\theta(\overline{t})&{\rm otherwise}.
     \end{array}\right.
  \]
 \end{alemma}
 \begin{proof}
  By \cite[Theorem 10.49]{RW98}, we have
  $\partial(\theta^2)(\overline{t})=D^*\theta(\overline{t})(\theta(\overline{t}))$
  where $D^*\theta(\overline{t})\!:\mathbb{R}\rightrightarrows\mathbb{R}$
  denotes the coderivative of the function $\theta^2$ at $\overline{t}$.
  From \cite[Proposition 9.24(b)]{RW98}, it follows that
  \(
    D^*\theta(\overline{t})(\theta(\overline{t}))
    =\partial(\theta(\overline{t})\theta)(\overline{t}),
  \)
  which implies the desired result.
 \end{proof}

 In order to complete the proof of Theorem \ref{error-bound1},
 we introduce the following notation
 \begin{subequations}
 \begin{align}\label{zkxk}
  v^k\!:=e-w^k,\ z^k\!:=Ax^k\!-b,\ \Delta x^k:=x^k-x^*\quad\ \forall k\in\mathbb{N};\\
  \label{xik}
  \xi^k\!:=(\gamma_{1,k-1}I+\gamma_{2,k-1}A^{\mathbb{T}}A)(x^{k-1}\!- x^{k})-\mu x^*
  \quad\ \forall k\in\mathbb{N}.
 \end{align}
 \end{subequations}
 With these notation, we can establish the following important technical lemma.
%---------------------------------------------------------------------------------
 \begin{alemma}\label{xk-lemma1}
  Suppose for some $k\ge 1$ there exists $S^{k-1}\supseteq S^*$ with
  \(
    \max_{i\in(S^{k-1})^c}w_i^{k-1}\le\frac{1}{2}.
  \)
  Then, when $\lambda\ge 8n^{-1}\!\interleave\!A_{\mathcal{I}\cdot}\!\interleave_1+8\|\xi^k\|_\infty$,
  it holds that $\big\|\Delta x^{k}_{(S^{k-1})^c}\|_1\le 3\|\Delta x^{k}_{S^{k-1}}\|_1$.
 \end{alemma}
 \begin{proof}
  Since $x^*$ is a feasible solution and $x^k$ is an optimal one to \eqref{subprob-x}, respectively,
  from the strong convexity of the objective function of \eqref{subprob-x},
  it follows that
  \begin{align*}
   & f(\!A x^*\!-\!b)+\lambda\langle v^{k-1},| x^*|\rangle+\frac{\mu}{2}\|x^*\|^2
    +\frac{\gamma_{1,k-1}}{2}\|x^*\!-\!x^{k-1}\|^2
    +\frac{\gamma_{2,k-1}}{2}\|A(x^*\!-\!x^{k-1})\|^2\\
  &\ge f(\!A x^k\!-\!b)+\lambda\langle v^{k-1},| x^k|\rangle+\frac{\mu}{2}\|x^{k}\|^2
      +\frac{\gamma_{1,k-1}}{2}\| x^k\!-\!x^{k-1}\|^2+\frac{\gamma_{2,k-1}}{2}\|A(x^k\!-\!x^{k-1})\|^2\\
  &\quad +\frac{1}{2}\langle x^*\!-x^k,(\mu I+\gamma_{1,k}I+\gamma_{2,k}A^{\mathbb{T}}A)(x^*\!-x^k)\rangle,
  \end{align*}
  which by $\frac{\mu}{2}\big[\|x^k\|^2-\|x^*\|^2\big]
   =\frac{\mu}{2}\|x^k\!-\!x^*\|^2+\mu\langle x^k\!-x^*,x^*\rangle$
   is equivalent to saying that
  \begin{align}\label{ineq1-ftau}
   &f(Ax^{k}\!-b)-f(Ax^*\!-b)+\mu\|\Delta x^k\|^2\nonumber\\
   &\le \lambda\langle v^{k-1},|x^*|-|x^k|\rangle+\gamma_{1,k-1}\langle x^{k-1}\!-\!x^{k}, x^k\!-x^*\rangle\nonumber\\
   &\quad +\gamma_{2,k-1}\langle A^\mathbb{T}A( x^{k-1}\!-\!x^{k}), x^k\!-x^*\rangle
     +\langle x^*\!-x^k,\mu x^*\rangle\nonumber\\
   &=\lambda\langle v^{k-1},| x^*|-| x^{k}|\rangle+\langle\xi^k, x^k\!-x^*\rangle
  \end{align}
  where the equality is due to the definition of $\xi^k$.
  By using $\theta(0)=0$ and equation \eqref{subdiff-theta},
  \begin{equation}\label{theta-ineq}
   \theta(z_i)\le\widetilde{\tau}\|z\|_\infty\ \
   {\rm for}\ \ i=1,\ldots,n\quad\ \forall z\in\mathbb{R}^n.
  \end{equation}
  Recall $0\ne\varpi=b-Ax^*$ and the definition of the index set $\mathcal{I}$.
  For each $k\in\mathbb{N}$, define $\mathcal{J}_k:=\big\{i\notin\mathcal{I}\!: z_i^{k}\neq0\big\}$. Then,
  together with the expression of $f$ and \eqref{theta-ineq}, we have
  \begin{align}\label{ineq2-ftau}
   f(A x^{k}\!-b)-f(A x^*\!-b)
   &=\frac{1}{n}\bigg[\sum_{i\in \mathcal{J}_k}\frac{\theta ^2(z^{k}_i)-\theta ^2(\varpi_i)}
    {\theta (z^{k}_i)+\theta (\varpi_i)}
   +\sum_{i\in \mathcal{I}}\frac{\theta ^2(z^{k}_i)-\theta ^2(\varpi_i)}
    {\theta (z^{k}_i)+\theta (\varpi_i)}\bigg]\nonumber\\
   &\ge\frac{1}{n}\bigg[\sum_{i\in \mathcal{J}_k}\frac{\theta ^2(z^{k}_i)-\theta ^2(\varpi_i)}
    {\widetilde{\tau}\|z^{k}\|_{\infty}}
   +\sum_{i\in \mathcal{I}}\frac{\theta ^2(z^{k}_i)-\theta ^2(\varpi_i)}
    {\theta (z^{k}_i)+\theta (\varpi_i)}\bigg].
  \end{align}
  Fix an arbitrary $\eta_i\in\partial(\theta^2)(\varpi_i)$.
  Since $\theta^2$ is strongly convex of modulus $\tau$, it holds that
  \begin{equation}\label{ineq-theta-tau}
    \theta^2(z^{k}_i)-\theta ^2(\varpi_i)
    \ge \eta_i(z_i^k-\varpi_i)
    +0.5\tau(z^{k}_i-\varpi_i)^2
    \ \ {\rm for}\ \ i=1,\ldots,n.
  \end{equation}
  This by Lemma \ref{theta-lemma} implies that
  $\theta^2(z^{k}_i)-\theta^2(\varpi_i)\ge\frac{\tau}{2}(z^{k}_i-\varpi_i)^2$
  for each $i\in\mathcal{J}_k$, and hence
  \begin{equation}\label{ineq3-ftau}
   \sum_{i\in \mathcal{J}_k}\frac{\theta ^2(z^{k}_i)-\theta ^2(\varpi_i)}
    {\widetilde{\tau}\|z^{k}\|_{\infty}}
   \ge \frac{\tau}{2\widetilde{\tau}}
    \sum_{i\in \mathcal{J}_k}\frac{(z^{k}_i-\varpi_i)^2}{\|z^{k}\|_{\infty}}.
  \end{equation}
  For each $i\in\mathcal{I}$, write $\widetilde{z}_i^{k}=\frac{\eta_i}{\theta (z_i^{k})+\theta (\varpi_i)}$.
  By Lemma \ref{theta-lemma}, clearly, $0\le\widetilde{z}_i^k\le 1$ for each $i\in\mathcal{I}$.
  Together with the inequality \eqref{ineq-theta-tau}, it immediately follows that
 \begin{align}\label{ineq4-ftau}
   \sum_{i\in \mathcal{I}}\frac{\theta ^2(z^{k}_i)-\theta ^2(\varpi_i)}
    {\theta (z^{k}_i)+\theta (\varpi_i)}
   &\ge\sum_{i\in \mathcal{I}}\widetilde{z}_i^k(z_i^k-\varpi_i)+\frac{\tau}{2}
     \sum_{i\in \mathcal{I}}\frac{(z^{k}_i-\varpi_i)^2}{\theta (z^{k}_i)+\theta (\varpi_i)}\nonumber\\
   &\ge-\|\widetilde{z}^k\|_\infty\|[A(x^k\!- x^*)]_{\mathcal{I}}\|_1+\frac{\tau}{2}
     \sum_{i\in \mathcal{I}}\frac{(z^{k}_i-\varpi_i)^2}{\widetilde{\tau}(\|z^{k}\|_{\infty}+\|\varpi\|_\infty)}\nonumber\\
   &\ge -\big\|[A(x^k\!- x^*)]_{\mathcal{I}}\big\|_{1}+\frac{\tau}{2\widetilde{\tau}}
     \sum_{i\in \mathcal{I}}\frac{(z^{k}_i-\varpi_i)^2}{\|z^{k}\|_{\infty}+\|\varpi\|_\infty}
 \end{align}
 where the second inequality is due to \eqref{theta-ineq}.
 Substituting \eqref{ineq3-ftau}-\eqref{ineq4-ftau} into \eqref{ineq2-ftau} yields that
 \begin{align*}
  f(A x^{k}\!-b)-f(A x^*\!-b)
  &\ge -\frac{1}{n}\|[A(x^k\!- x^*)]_{\mathcal{I}}\|_1+\frac{\tau }{2n\widetilde{\tau}}
     \sum_{i\in \mathcal{J}_k\cup \mathcal{I}}\frac{(z^{k}_i-\varpi_i)^2}
     {\|z^{k}\|_{\infty}\!+\!\|\varpi\|_\infty}\nonumber\\
     &= -\frac{1}{n}\|[A(x^k\!- x^*)]_{\mathcal{I}}\|_1+\frac{\tau \|A( x^k\!- x^*)\|^2}{2n\widetilde{\tau}(\|z^{k}\|_{\infty}\!+\!\|\varpi\|_\infty)}.
  \end{align*}
 By combining this inequality and the inequality \eqref{ineq1-ftau}, it follows that
 \begin{align}\label{f-mainineq}
  &\mu\|\Delta x^k\|^2
  +\frac{\tau \|A( x^k- x^*)\|^2}{2n\widetilde{\tau}(\|z^{k}\|_{\infty}\!+\!\|\varpi\|_\infty)}\nonumber\\
  &\le \lambda\langle v^{k-1},| x^*|-| x^{k}|\rangle
       +\frac{1}{n}\big\|[A(x^k\!- x^*)]_{\mathcal{I}}\big\|_{1}+\langle \xi^k, x^k\!- x^*\rangle\nonumber\\
  &\le \lambda\langle v^{k-1},| x^*|-| x^{k}|\rangle
       +\big(n^{-1}\!\interleave\!A_{\mathcal{I}\cdot}\!\interleave_1
       +\|\xi^k\|_{\infty}\big)\big\|x^k\!- x^*\big\|_1\nonumber\\
%  &\le \lambda\Big(\textstyle{\sum_{i\in S^*}}v_i^{k-1}|\Delta x_i^k|
%       -\textstyle{\sum_{i\in (S^{k-1})^c}}v_i^{k-1}|\Delta x_i^k|\Big)\nonumber\\
%  &\quad\ +\big(n^{-1}\|A\|_1+\|\xi^k\|_{\infty}\big)\| x^k\!- x^*\|_1\nonumber\\
  &\le\lambda\Big(\textstyle{\sum_{i\in S^*}}v_i^{k-1}|\Delta x_i^k|
       -\textstyle{\sum_{i\in (S^{k-1})^c}}v_i^{k-1}|\Delta x_i^k|\Big)\nonumber\\
  & \quad\ +\big(n^{-1}\!\interleave\!A_{\mathcal{I}\cdot}\!\interleave_1+\|\xi^k\|_{\infty}\big)
   \big(\|\Delta x_{S^{k-1}}^k\|_{1}+\|\Delta x_{(S^{k-1})^{c}}^k\|_{1}\big).
 \end{align}
  Since $S^{k-1}\supseteq S^*$ and $v_i^{k-1}\in[0.5,1]$ for $i\in(S^{k-1})^{c}$,
  from the last inequality we have
 \begin{align*}
  \mu\|\Delta x^k\|^2+\frac{\tau \|A(x^k\!- x^*)\|^2}{2n\widetilde{\tau}(\|z^{k}\|_{\infty}\!+\!\|\varpi\|_\infty)}
  &\le\textstyle{\sum_{i\in S^{k-1}}}\big(\lambda v_i^{k-1}
     +n^{-1}\!\interleave\!A_{\mathcal{I}\cdot}\!\interleave_1+\|\xi^k\|_\infty\big)\big|\Delta x_i^k\big|\nonumber\\
  &\quad +\textstyle{\sum_{i\in (S^{k-1})^c}}
    \big(n^{-1}\!\interleave\!A_{\mathcal{I}\cdot}\!\interleave_1+\|\xi^k\|_\infty-\lambda/2\big)\big|\Delta x_i^k\big|\\
  &=\big(\lambda +n^{-1}\!\interleave\!A_{\mathcal{I}\cdot}\!\interleave_1+\|\xi^k\|_\infty\big)
   \big\|\Delta x_{S^{k-1}}^k\big\|_1\nonumber\\
  &\quad +\big(n^{-1}\!\interleave\!A_{\mathcal{I}\cdot}\!\interleave_1+\|\xi^k\|_\infty-\lambda/2\big)
  \big\|\Delta x_{(S^{k-1})^{c}}^k\big\|_1.\nonumber
 \end{align*}
 From the nonnegativity of the left hand side and the given assumption on $\lambda$,
 we have
 \[
   \big\|\Delta x_{(S^{k-1})^{c}}^k\big\|_1
   \le\frac{\lambda +n^{-1}\!\interleave\!A_{\mathcal{I}\cdot}\!\interleave_1+\|\xi^k\|_\infty}
   {0.5\lambda-n^{-1}\!\interleave\!A_{\mathcal{I}\cdot}\!\interleave_1-\|\xi^k\|_\infty}
   \big\|\Delta x_{S^{k-1}}^k\big\|_1
  \le 3\big\|\Delta x_{S^{k-1}}^k\big\|_1.
 \]
 This implies that the desired result holds. The proof is completed.
 \end{proof}

 Now by using the inequality \eqref{f-mainineq} and Lemma \ref{xk-lemma1},
 we obtain the following conclusion.
 %-----------------------------------------------------------------------------------
 \begin{alemma}\label{xk-lemma2}
  Suppose that the matrix $A^{\mathbb{T}}A/n$ satisfies the RE condition of parameter
  $\kappa>0$ over $\mathcal{C}(S^*)$ and for some $k\ge 1$ there exists an index set $S^{k-1}$
  with $|S^{k-1}|\le 1.5s^*$ such that $S^{k-1}\supseteq S^*$ and $\max_{i\in(S^{k-1})^c}w_i^{k-1}\le\frac{1}{2}$.
  If the parameter $\lambda$ is chosen such that $8n^{-1}\!\interleave\!A_{\mathcal{I}\cdot}\!\interleave_1+8\|\xi^k\|_\infty\le\lambda
  <\frac{2\mu\widetilde{\tau}\|\varpi\|_\infty+\tau\kappa-4\widetilde{\tau}\|A\|_{\infty}
 (n^{-1}\interleave A_{\mathcal{I}\cdot}\interleave_1+\|\xi^k\|_\infty)|S^{k-1}|}
 {4\widetilde{\tau}\|A\|_{\infty}\|v_{S^*}^{k-1}\|_{\infty}|S^{k-1}|}$, then
  \[
    \big\|\Delta x^{k}\big\|\le\frac{2\widetilde{\tau}\|\varpi\|_{\infty}\big(\lambda\|v_{S^*}^{k-1}\|_{\infty}
       +n^{-1}\!\interleave\!A_{\mathcal{I}\cdot}\!\interleave_1+\|\xi^k\|_{\infty}\big)
    \sqrt{|S^{k-1}|}}
     {2\mu\widetilde{\tau}\|\varpi\|_{\infty}+\tau\kappa
        -4\widetilde{\tau}\|A\|_{\infty}\big(\lambda\|v_{S^*}^{k-1}\|_{\infty}
       +n^{-1}\!\interleave\!A_{\mathcal{I}\cdot}\!\interleave_1+\|\xi^k\|_{\infty}\big)|S^{k-1}|}.
  \]
 \end{alemma}
 \begin{proof}
  Note that $\|z^k\|_\infty+\|\varpi\|_\infty=\|\varpi-\!A\Delta x^k\|_{\infty}+\|\varpi\|_\infty
  \le\|A\Delta x^k\|_{\infty}+2\|\varpi\|_\infty$. So,
  \[
   \frac{\tau\|A(x^k-x^*)\|^2}{2n\widetilde{\tau}(\|z^{k}\|_{\infty}+\|\varpi\|_\infty)}
   \ge\frac{\tau\|A\Delta x^k\|^2}{2n\widetilde{\tau}(\|A\Delta x^k\|_{\infty}+2\|\varpi\|_\infty)}.
  \]
  Together with the inequality \eqref{f-mainineq} and
  $v_i^{k-1}\in[0.5,1]$ for $i\in(S^{k-1})^{c}$, it follows that
 \begin{align*}
  &\mu\|\Delta x^k\|^2+\frac{\tau\|A\Delta x^k\|^2}
   {2n\widetilde{\tau}(\|A\Delta x^k\|_{\infty}+2\|\varpi\|_\infty)}\nonumber\\
  &\le\lambda{\textstyle\sum_{i\in S^*}}v_i^{k-1}|\Delta x_i^k|
           -({\lambda}/{2}){\textstyle\sum_{i\in (S^{k-1})^c}}|\Delta x_i^k|\nonumber\\
  &\quad +\big(n^{-1}\!\interleave\!A_{\mathcal{I}\cdot}\!\interleave_1+\|\xi^k\|_{\infty}\big)
   \big(\|\Delta x_{S^{k-1}}^k\|_{1}+\|\Delta x_{(S^{k-1})^{c}}^k\|_{1}\big)\nonumber\\
  &\le \big(\lambda\|v_{S^*}^{k-1}\|_{\infty}+n^{-1}\!\interleave\!A_{\mathcal{I}\cdot}\!\interleave_1
     +\|\xi^k\|_{\infty}\big)\|\Delta x_{S^{k-1}}^k\|_{1}
 \end{align*}
 where the last inequality is due to
 $\lambda\ge 8n^{-1}\!\interleave\!A_{\mathcal{I}\cdot}\!\interleave_1+8\|\xi^k\|_\infty$.
 By Lemma \ref{xk-lemma1}, we know that
 $\|\Delta x^{k}_{(S^{k-1})^c}\|_1\le 3\|\Delta x^{k}_{S^{k-1}}\|_1$,
 which by the given assumption means that $\Delta x^{k}\in\mathcal{C}(S^*)$.
 From the given assumption on $\frac{1}{n}A^{\mathbb{T}}A$, we have
 $\|A\Delta x^k\|^2\ge 2n\kappa\|\Delta x^k\|^2$. Then,
 \[
  \mu\|\Delta x^k\|^2+\frac{\tau\kappa\|\Delta x^k\|^2}
   {\widetilde{\tau}(\|A\Delta x^k\|_{\infty}\!+2\|\varpi\|_\infty)}
  \le\Big(\lambda\|v_{S^*}^{k-1}\|_{\infty}+\frac{1}{n}\!\interleave\!A_{\mathcal{I}\cdot}\!\interleave_1
     +\|\xi^k\|_{\infty}\Big)\big\|\Delta x_{S^{k-1}}^k\big\|_1.
 \]
  Multiplying the both sides of this inequality by
  $\widetilde{\tau}(\|A\Delta x^k\|_{\infty}+2\|\varpi\|_\infty)$ yields that
  \begin{align*}
   &\big[\mu\widetilde{\tau}(\|A\Delta x^k\|_{\infty}+2\|\varpi\|_\infty)+\tau\kappa\big]\|\Delta x^k\|^2\\
   &\le \widetilde{\tau}\big(\|A\Delta x^k\|_{\infty}+2\|\varpi\|_\infty\big)
       \big(\lambda\|v_{S^*}^{k-1}\|_{\infty}+n^{-1}\!\interleave\!A_{\mathcal{I}\cdot}\!\interleave_1
     +\|\xi^k\|_{\infty}\big)\big\|\Delta x_{S^{k-1}}^k\big\|_1\\
   &\le \widetilde{\tau}\|A\Delta x^k\|_{\infty}\big(\lambda\|v_{S^*}^{k-1}\|_{\infty}
       +n^{-1}\!\interleave\!A_{\mathcal{I}\cdot}\!\interleave_1+\|\xi^k\|_{\infty}\big)
       \big\|\Delta x_{S^{k-1}}^k\big\|_1\\
   &\quad + 2\widetilde{\tau}\|\varpi\|_\infty
      \big(\lambda\|v_{S^*}^{k-1}\|_{\infty}
       +n^{-1}\!\interleave\!A_{\mathcal{I}\cdot}\!\interleave_1+\|\xi^k\|_{\infty}\big)
       \big\|\Delta x_{S^{k-1}}^k\big\|_1.
  \end{align*}
  Note that $\|A\Delta x^k\|_{\infty}\le \|A\|_{\infty}\|\Delta x^k\|_1$.
  Along with $\|\Delta x^{k}_{(S^{k-1})^c}\|_1\le 3\|\Delta x^{k}_{S^{k-1}}\|_1$,
  we have $\|A\Delta x^k\|_{\infty}\le 4\|A\|_{\infty}\|\Delta x_{S^{k-1}}^k\|_1$.
  So, the right hand side of the last inequality satisfies
  \begin{align*}
   {\rm RHS}
   &\le 4\widetilde{\tau}\|A\|_{\infty}\big(\lambda\|v_{S^*}^{k-1}\|_{\infty}
       +n^{-1}\!\interleave\!A_{\mathcal{I}\cdot}\!\interleave_1+\|\xi^k\|_{\infty}\big)
       \big\|\Delta x_{S^{k-1}}^k\big\|_1^2\\
   &\quad +2\widetilde{\tau}\|\varpi\|_{\infty}\big(\lambda\|v_{S^*}^{k-1}\|_{\infty}
       +n^{-1}\!\interleave\!A_{\mathcal{I}\cdot}\!\interleave_1+\|\xi^k\|_{\infty}\big)
      \big\|\Delta x_{S^{k-1}}^k\big\|_1\\
   &\le 4\widetilde{\tau}\|A\|_{\infty}\big(\lambda\|v_{S^*}^{k-1}\|_{\infty}
       +n^{-1}\!\interleave\!A_{\mathcal{I}\cdot}\!\interleave_1+\|\xi^k\|_{\infty}\big)
       |S^{k-1}|\big\|\Delta x_{S^{k-1}}^k\big\|^2\\
   &\quad +2\widetilde{\tau}\|\varpi\|_{\infty}\big(\lambda\|v_{S^*}^{k-1}\|_{\infty}
       +n^{-1}\!\interleave\!A_{\mathcal{I}\cdot}\!\interleave_1+\|\xi^k\|_{\infty}\big)
    \sqrt{|S^{k-1}|}\big\|\Delta x_{S^{k-1}}^k\big\|.
  \end{align*}
  From the last two equations, a suitable rearrangement yields that
 \begin{align*}
  &\Big[2\mu\widetilde{\tau}\|\varpi\|_{\infty}+\tau\kappa
        -4\widetilde{\tau}\|A\|_{\infty}\big(\lambda\|v_{S^*}^{k-1}\|_{\infty}
       +n^{-1}\!\interleave\!A_{\mathcal{I}\cdot}\!\interleave_1+\|\xi^k\|_{\infty}\big)|S^{k-1}|\Big]
       \|\Delta x^k\|^2\\
  &\le 2\widetilde{\tau}\|\varpi\|_{\infty}\big(\lambda\|v_{S^*}^{k-1}\|_{\infty}
       +n^{-1}\!\interleave\!A_{\mathcal{I}\cdot}\!\interleave_1+\|\xi^k\|_{\infty}\big)
    \sqrt{|S^{k-1}|}\big\|\Delta x_{S^{k-1}}^k\big\|,
 \end{align*}
 which by $\lambda<\frac{2\mu\widetilde{\tau}\|\varpi\|_\infty+\tau\kappa-4\widetilde{\tau}\|A\|_{\infty}
 (n^{-1}\interleave A_{\mathcal{I}\cdot}\interleave_1+\|\xi^k\|_\infty)|S^{k-1}|}
 {4\widetilde{\tau}\|A\|_{\infty}\|v_{S^*}^{k-1}\|_{\infty}|S^{k-1}|}$
 implies the desired result.
 \end{proof}

 \medskip
 \noindent
 {\bf\large Proof of Theorem \ref{error-bound1}:}
 \begin{aproof}
  Since $x^k\to\overline{x}$ as $k\to\infty$ and $\gamma_{1,k}\ge\underline{\gamma_1}$
  and $\gamma_{2,k}\ge\underline{\gamma_2}$, by the definition of $\xi^k$ in \eqref{xik},
  we have $\xi^k\to\mu x^*$. So, there exists $\widehat{k}\in\mathbb{N}$ such that
  $\|\xi^k\|_\infty\le\frac{3}{2}\mu\|x^*\|_\infty$ for all $k\ge\overline{k}$.
  Since $x^k\to\overline{x}$, there is $\widetilde{k}\in\mathbb{N}$
  such that for all $k\ge\widetilde{k}$,
  \[
    |x_i^k|-|\overline{x}_i|\le|x_i^k-\overline{x}_i|\le\frac{1}{\rho(a+1)}
    \ \ {\rm for}\ i=1,2,\ldots,n.
  \]
  This, by the assumption on $\overline{x}$, implies that
  $|x_i^k|\le\frac{1}{\rho}$ for $i\notin S^*$, and from \eqref{wik}
  we have $w_i^k\in[0,1/2]$ for each $i\in(S^*)^{c}$ when $k\ge\widetilde{k}$.
  Set $\overline{k}:=\max(\widehat{k},\widetilde{k})$ and for each $k$
  define $S^{k-1}:=S^*\cup\{i\notin S^*\!:w^{k-1}_i>\frac{1}{2}\}$.
  If $|S^{k-1}|\leq1.5s^*$ for $k\ge\overline{k}$,
  from Lemma \ref{xk-lemma2} we have
  \begin{align}\label{equa-41}
   \big\| x^{k}\!- x^*\big\|
   &\le\frac{2\widetilde{\tau}\big(\lambda\|v_{S^*}^{k-1}\|_{\infty}
      +n^{-1}\!\interleave\!A_{\mathcal{I}\cdot}\!\interleave_1+\frac{3}{2}\mu\|x^*\|_\infty\big)
      \sqrt{|S^{k-1}|}\|\varpi\|_\infty}
     {2\mu\widetilde{\tau}\|\varpi\|_\infty+\tau\kappa -4\widetilde{\tau}\|A\|_{\infty}\big(\lambda\|v_{S^*}^{k-1}\|_{\infty}
      +n^{-1}\!\interleave\!A_{\mathcal{I}\cdot}\!\interleave_1+\frac{3}{2}\mu\|x^*\|_\infty\big)|S^{k-1}|}\nonumber\\
   &\le\frac{2\widetilde{\tau}\big(\lambda\|v_{S^*}^{k-1}\|_{\infty}
      +n^{-1}\!\interleave\!A_{\mathcal{I}\cdot}\!\interleave_1+\frac{3}{2}\mu\|x^*\|_\infty\big)
      \sqrt{|S^{k-1}|}\|\varpi\|_\infty}
     {2\mu\widetilde{\tau}\|\varpi\|_\infty+\tau\kappa -4\widetilde{\tau}\|A\|_{\infty}\big(\lambda\|v_{S^*}^{k-1}\|_{\infty}
      +n^{-1}\!\interleave\!A_{\mathcal{I}\cdot}\!\interleave_1+\frac{3}{2}\mu\|x^*\|_\infty\big)\sqrt{1.5s^*}}\nonumber\\
   &\le c\widetilde{\tau}\big(\lambda\|v_{S^*}^{k-1}\|_{\infty}
      +n^{-1}\!\interleave\!A_{\mathcal{I}\cdot}\!\interleave_1+1.5\mu\|x^*\|_\infty\big)\sqrt{|S^{k-1}|}\|\varpi\|_\infty\\
   &\le \frac{9.5c\widetilde{\tau}\lambda\|\varpi\|_\infty}{8}\sqrt{1.5s^*}\nonumber
  \end{align}
  where the third inequality is due to the restriction on $\lambda$,
  and the last one is since
  $n^{-1}\!\interleave\!A_{\mathcal{I}\cdot}\!\interleave_1+\mu\|x^*\|_\infty
  \le\!\frac{\lambda}{8}$
  and $\|v_{S^*}^{k-1}\|_{\infty}\le 1$. Taking the limit $k\rightarrow\infty$
  to the both sides of \eqref{equa-41}, we obtain the desired result.
  So, it suffices to argue that $|S^{k-1}|\le1.5s^*$ for all $k\ge\overline{k}$.
  When $k=\overline{k}$, this statement holds by the above discussions.
  Assuming that $|S^{k-1}|\le1.5s^*$ holds for $k=l$ with $l\ge\overline{k}$,
  we prove that it holds for $k=l+1$. Indeed,
  since $S^{l}\setminus S^*=\big\{i\notin S^*\!:w_i^{l}>\frac{1}{2}\big\}$,
  we have $w_i^{l}\in(\frac{1}{2},1]$ for $i\in S^{l}\setminus S^*$. Together with
  the formula \eqref{wik}, we deduce that $\rho| x_i^{l}|\ge 1$, and hence
  the following inequality holds:
  \[
    \sqrt{|S^{l}\setminus S^*|}\le\sqrt{\sum_{i\in S^{l}\setminus S^*}\rho^2| x_i^{l}|^2}
    =\sqrt{\sum_{i\in S^{l}\setminus S^*}\rho^2| x_i^{l}- x_i^*|^2}.
  \]
  Since the statement holds for $k=l$, it holds that
  $\|x^{l}-x^*\|\le\frac{9.5c\widetilde{\tau}\lambda\|\varpi\|_\infty}{8}\sqrt{1.5s^*}$.
  Thus,
  \begin{equation}\label{Sl-equa}
    \sqrt{|S^{l}\setminus S^*|}\le \rho\| x^{l}- x^*\|
    \le \frac{9.5c\widetilde{\tau}\rho\lambda\|\varpi\|_\infty}{8}\sqrt{1.5s^*}
    \le \sqrt{0.5s^*}
  \end{equation}
  where the last inequality is due to $\rho\lambda\le\frac{8}{9.5\sqrt{3}c\widetilde{\tau}\|\varpi\|_\infty}$.
  The inequality \eqref{Sl-equa} implies that $|S^{l}|\leq1.5s^*$.
  This shows that the stated inequality $|S^{l}|\le1.5s^*$ holds.
 \end{aproof}

  \bigskip
 \noindent
 {\bf\large Appendix C}

 \medskip
 \noindent
 {\bf\large Proof of Proposition \ref{x0err-bound}:}
  Write $\Delta x^{0}:=x^0-x^*$. From the given condition,
  the strong convexity of $\Theta$ and the fact that $x^*\in\Omega$,
  it follows that
  \begin{align}\label{first-ineq}
   &f(A x^*\!-b)+\lambda\|x^*\|_1+({\gamma_{1,0}}/{2})\| x^*\|^2
      +({\gamma_{2,0}}/{2})\|\!A x^*-b\|^2\nonumber\\
  &\ge f(A x^0\!-b)+\lambda\|x^0\|_1+({\gamma_{1,0}}/{2})\| x^0\|^2
      +({\gamma_{2,0}}/{2})\|A x^0\!-b\|^2 \nonumber\\
  &\qquad +\langle\xi^0, x^*\!-\! x^0\rangle
      +0.5\langle (x^*\!-\!x^0), (\gamma_{1,0}I+\gamma_{20}A^{\mathbb{T}}A)(x^*\!-\!x^0)\rangle.
  \end{align}
  From $f(z)=\frac{1}{n}\sum_{i=1}^n\theta(z_i)$ and Assumption \ref{theta-assump},
  it is not difficult to obtain that
  \[
    f(A x^0\!-b)-f(A x^*\!-b)\ge -({\widetilde{\tau}}/{n})\|A(x^*\!-\!x^0)\|_1.
  \]
  A simple calculation, together with $b=Ax^*+\varpi$, immediately yields that
  \begin{align*}
    \frac{1}{2}\|x^0\|^2-\frac{1}{2}\| x^*\|^2
   =\frac{1}{2}\|x^0\!-\!x^*\|^2+\langle x^0-x^*,x^*\rangle,\qquad\qquad\qquad\\
     \frac{1}{2}\|A x^0\!-b\|^2-\frac{1}{2}\|A x^*\!-b\|^2
     =\langle x^*\!-\!x^0,A^{\mathbb{T}}\varpi\rangle
      +\frac{1}{2}\langle x^0\!-\!x^*,A^{\mathbb{T}}A(x^0\!-\!x^*)\rangle.
   \end{align*}
  By combining the last three equations with \eqref{first-ineq}
  and using $\|\xi^0\|_\infty\le\epsilon$, we obtain that
  \begin{align*}
  \gamma_{1,0}\|\Delta x^0\|^2
  &\le \lambda(\|x^*\|_1\!-\|x^{0}\|_1)+n^{-1}{\widetilde{\tau}}\|A(x^*\!-\!x^0)\|_1
       +\langle x^0\!-x^*,\xi^0\!+\gamma_{20}A^{\mathbb{T}}\varpi\!-\gamma_{10}x^*\rangle\\
  &\le \lambda\Big(\textstyle{\sum_{i\in S^*}}|\Delta x_i^0|
       -\textstyle{\sum_{i\in (S^{*})^c}}|\Delta x_i^0|\Big) \nonumber\\
  &\quad +\big[n^{-1}{\widetilde{\tau}}\!\interleave\!A\!\interleave_1+\gamma_{1,0}\|x^*\|_{\infty}
       +\gamma_{2,0}\|A^{\mathbb{T}}\varpi\|_{\infty}+\epsilon\big]\|x^0\!-\!x^*\|_1\\
  %&\le\textstyle{\sum_{i\in S^{*}}}\big(\lambda +n^{-1}\widetilde{\tau}\|A\|_1
%  +\gamma_{1,0}\| x^*\|_{\infty}+\gamma_{2,0}\|A^{\mathbb{T}}\varpi\|_{\infty}+\epsilon\big)\big|
%  \Delta x_i^0\big|\nonumber\\
%  &\quad +\textstyle{\sum_{i\in (S^{*})^c}}\big(n^{-1}\widetilde{\tau}\|A\|_1+\gamma_{1,0}
%  \| x^*\|_{\infty}+\gamma_{2,0}\|A^{\mathbb{T}}\varpi\|_{\infty}+r_0-\lambda\big)
%  \big|\Delta x_i^0\big|\nonumber\\
  &\le\big(\lambda +n^{-1}\widetilde{\tau}\!\interleave\!A\!\interleave_1+\gamma_{1,0}
   \| x^*\|_{\infty}+\gamma_{2,0}\|A^{\mathbb{T}}\varpi\|_{\infty}+\epsilon\big)
   \|\Delta x_{S^{*}}^0\|_1\nonumber\\
  &\quad +\big(n^{-1}\widetilde{\tau}\!\interleave\!A\!\interleave_1+\gamma_{1,0}
   \| x^*\|_{\infty}+\gamma_{2,0}\|A^{\mathbb{T}}\varpi\|_{\infty}+\epsilon-\lambda\big)
   \|\Delta x_{(S^{*})^{c}}^0\|_1.\nonumber
  \end{align*}
  Along with the given assumption on $\lambda$ and the nonnegativity
  of $\|\Delta x^0\|^2$, it follows that
  \(
    \|\Delta x_{(S^{*})^{c}}^0\|_1\le 3\|\Delta x_{S^{*}}^0\|_1.
  \)
  By combining this with the last inequality, we have
  \[
    \gamma_{1,0}\|\Delta x^0\|^2
    \le\big(\lambda +n^{-1}\widetilde{\tau}\!\interleave\!A\!\interleave_1+\gamma_{1,0}
   \| x^*\|_{\infty}+\gamma_{2,0}\|A^{\mathbb{T}}\varpi\|_{\infty}+\epsilon\big)
    \big\|\Delta x_{S^{*}}^0\big\|_1\le\frac{3\lambda\sqrt{s^*}}{2}\big\|\Delta x^0\big\|
  \]
  which implies the desired conclusion. The proof is then completed.

%   \begin{subequations}
%   \begin{align*}
%   \varphi_{\rho}(t)=\left\{\begin{array}{cl}
%                      0 & \textrm{if}\ |t|\leq \frac{2}{\rho(a+1)},\\
%                      \frac{((a+1)\rho|t|-2)^2}{4(a^2-1)} & \textrm{if}\ \frac{2}{\rho(a+1)}<|t|\leq \frac{2a}{\rho(a+1)},\\
%                      \rho|t|-1 & \textrm{if}\ |t|>\frac{2a}{\rho(a+1)}.
%                \end{array}\right.\\
%    \varphi_{\rho}'(t)=\left\{\begin{array}{cl}
%                      0 & \textrm{if}\ |t|\leq \frac{2}{\rho(a+1)};\\
%                      \frac{\rho((a+1)\rho|t|-2){\rm sign}(t)}{2(a-1)} & \textrm{if}\ \frac{2}{\rho(a+1)}<|t|\leq \frac{2a}{\rho(a+1)};\\
%                      \rho{\rm sign}(t) & \textrm{if}\ |t|>\frac{2a}{\rho(a+1)}.
%                \end{array}\right.
% \end{align*}
% \end{subequations}

\end{document}